\documentclass[11pt]{amsart}
\usepackage{amsmath}
\usepackage{amsfonts}
\usepackage{amssymb}
\usepackage{amsthm}
\usepackage{url}
\usepackage[all]{xy}
\usepackage{dsfont}
\usepackage{graphicx}
\usepackage{caption}
\usepackage{subcaption}
\usepackage{comment} 
\usepackage{stmaryrd}
\usepackage{hyperref}
\usepackage{todonotes}
\usepackage{color}
\usepackage{enumerate,tikz-cd}
\usepackage[backend=biber,style=ieee-alphabetic,maxbibnames=10,hyperref=true,doi=false,url=false,isbn=false,sorting=nty]{biblatex}
\usepackage{comment}
\allowdisplaybreaks

\numberwithin{equation}{section}

\newtheorem{proposition}{Proposition}[section]

\newtheorem{lemma}[proposition]{Lemma}
\newtheorem{theorem}[proposition]{Theorem}
\newtheorem{corollary}[proposition]{Corollary}

\theoremstyle{definition}
\newtheorem{remark}[proposition]{Remark}
\newtheorem{definition}[proposition]{Definition}
\newtheorem{example}[proposition]{Example}

\DeclareMathOperator{\GL}{GL}

\DeclareMathOperator{\Aut}{Aut}

\DeclareMathOperator{\Proj}{Proj}
\DeclareMathOperator{\Spec}{Spec}

\DeclareMathOperator{\ord}{ord}

\newcommand{\R}{\mathbb{R}}
\newcommand{\C}{\mathbb{C}}

\newcommand{\Z}{\mathbb{Z}}

\newcommand{\Q}{\mathbb{Q}}
\newcommand{\G}{\mathcal{G}}
\let\P\undefined
\newcommand{\P}{\mathbb{P}}
\renewcommand{\epsilon}{\varepsilon}

\newcommand{\M}{\mathcal{M}}

\let\O\undefined
\newcommand{\O}{\mathcal{O}}

\renewcommand{\L}{\mathcal{L}}
\newcommand{\X}{\mathcal{X}}
\newcommand{\Y}{\mathcal{Y}}

\renewcommand{\G}{\mathbb{G}}

\renewcommand{\phi}{\varphi}

\newcommand\cO{\mathcal{O}}

\newcommand\Hom{\mathrm{Hom}}

\title[Equivariant CM minimization for extremal manifolds]{Equivariant CM minimization for extremal manifolds}

\author[Gabriel Frey]{Gabriel Frey
}

\address{Gabriel Frey, Mathematics Institute, University of Warwick, Coventry CV4 7AL, United Kingdom}\email{gabriel.frey@warwick.ac.uk}
\begin{document}

\begin{abstract}

We prove an equivariant version of the CM minimization conjecture for extremal K\"ahler manifolds.
This involves proving that, given an equivariant punctured family of polarized varieties, a relative version of the CM degree is strictly minimized by an extremal filling.
This generalizes a result by Hattori for cscK manifolds with discrete automorphism group by allowing automorphisms and extremal metrics.
As a main tool, we extend results by Sz\'ekelyhidi on asymptotic filtration Chow stability of cscK manifolds with discrete automorphism group to the extremal setting.

\end{abstract}

\maketitle

\section{Introduction}

A central problem in complex and algebraic geometry is the construction of moduli spaces parameterizing projective varieties.
In our case of interest, we consider the moduli problem for polarized varieties; that is, projective varieties endowed with an ample line bundle.
The valuative criterion for separatedness ensures that a moduli space of a fixed class of polarized varieties is \textit{separated} provided that any family $(X,L) \to \Delta^*$ of polarized varieties of said class over the open punctured disc $\Delta^*$ in $\C$ has \textit{at most} one filling $(\mathcal{X},\mathcal{L}) \to \Delta$ to the whole open disk, up to isomorphism.

This class of problem may be approached numerically:
in a geometric invariant theory (GIT) setting, Wang--Xu \cite{WX12} proved that the semistable filling for family over a fixed base minimizes the GIT height, a numerical invariant introduced Wang \cite{wang2012height}.
It is known to experts that this can be improved to strict minimization for polystable fillings, giving a numerical proof of separatedness of GIT quotients.
Wang's work was motivated by Cornalba--Harris \cite{cornalba1988divisor}, Zhang \cite{zhang1996heights} and Xiao \cite{xiao1987fibered}, and raised the question of how to extend this theory to the setting of stability of varieties.

The analogue of GIT stability for polarized varieties is K-stability \cite{tian1997kahler,donaldson2002scalar}, where the numerical invariant generalizing Wang's GIT height is the degree of the Chow--Mumford (CM) line bundle.
In this case, the associated minimization problem is known as the \textit{CM minimization conjecture} for K-polystable fillings \cite{wang2012height,odaka2012on, WX12}.
The CM line bundle is a $\Q$-line bundle defined on the base of flat family of polarized varieties, originally introduced by Fujiki--Schumacher \cite{FS90} in their construction of the moduli space of constant scalar curvature K\"ahler (cscK) manifolds, and was further developed by Paul--Tian \cite{PT06}.

Hattori proved the following result on the degree of the CM line bundle $\mathrm{CM}(\mathcal{X},\mathcal{L})$ for cscK fillings $(\mathcal{X},\mathcal{L})$ over curves:
\begin{theorem}[{\cite[Theorem 1.2]{hattori2024minimizing}}]\label{thm : hattoris theorem}
    Let $C$ be a smooth projective curve with a special point $0$, and let $(\mathcal{X},\mathcal{L}) \to C$ and $(\mathcal{X}',\mathcal{L}') \to C$ be fillings of a flat family $(X,L) \to C \setminus \{ 0 \}$ such that the special fiber $(\mathcal{X}_0, \mathcal{L}_0)$ is smooth, admits a cscK metric and has discrete automorphism group.
    Then
    \begin{equation*}
        \mathrm{CM}(\mathcal{X},\mathcal{L})
        \leq 
        \mathrm{CM}(\mathcal{X}',\mathcal{L}'),
    \end{equation*}
    where the inequality is strict if and only if the fillings are not isomorphic.
\end{theorem}
The existence of cscK metrics is closely related to the stability of polarized manifolds, through the Yau--Tian--Donaldson conjecture \cite{yau1992open,tian1997kahler,donaldson2002scalar}, which predicts that a polarized manifold admits a cscK metric if and only if it is K-polystable.
Thus Theorem \ref{thm : hattoris theorem} may be viewed as a proof of the CM minimization conjecture, with an analytic stability hypothesis on the special fiber, under a further discrete automorphism group assumption.
The CM minimization conjecture was solved completely algebro-geometrically in the Fano case, by Blum--Xu \cite[Section 3]{blum2019uniqueness} and Xu \cite[Section 9.4]{xu2021k}.

As we shall see in our work, one may replace the projective curve with the formal disk as the base, in line with the valuative criterion for separatedness.
Since the disk is not proper, the degree of the CM line bundle for a model is not well-defined, but one may identify the difference of their degrees as a rational number.
Theorem \ref{thm : hattoris theorem} implies separatedness of the moduli space of polarized cscK manifolds with discrete automorphism group, a result originally proven by Fujiki--Schumacher \cite{FS90}, and which also follows from Donaldson's result \cite{donaldson2001scalar} on asymptotic Chow stability of cscK manifolds with discrete automorphism group, using GIT.


\subsection{Main results}

Let $R$ be a discrete valuation ring (DVR) with fraction field $K$, whose residue field is the field of complex numbers $\C$.
The spectrum $\Spec R$ consists of a closed point $0$ and a generic point $\Spec K$.
The guiding example is the ring of formal power series $R = \C[\![ t ]\!]$, whose spectrum is called the \textit{formal disk} and $K = \C(\!( t )\!)$ is the field of formal Laurent series.

Let $(X,L)$ be a polarized $K$-variety, which we think of as a degeneration of a polarized variety over $\Spec K$, and choose a maximal torus $T$ in the automorphism group $\mathrm{Aut}_K(X,L)$.
The main result of the paper is the following generalization of Theorem \ref{thm : hattoris theorem}, allowing automorphisms and extremal metrics (i.e. K\"ahler metrics $\omega$ satisfying $\overline{\partial} \, \mathrm{grad}^{1,0}_\omega \mathrm{scal}(\omega) = 0$, see \cite{calabi1982extremal}):
\begin{theorem}\label{thm : main theorem}
    Let $(\mathcal{X},\mathcal{L})$ and $(\mathcal{X}',\mathcal{L}')$ be $T$-equivariant models of $(X,L)$ such that the restricted torus $T_0$ in $\mathrm{Aut}(\mathcal{X}_0, \mathcal{L}_0)$ is maximal. If the special fiber $(\mathcal{X}_0, \mathcal{L}_0)$ is smooth and admits an extremal metric, then there exists a rational one-parameter subgroup $\xi$ of $T$ such that
    \begin{equation*}
        \mathrm{CM}(\mathcal{X}_\xi,\mathcal{L}_\xi) \leq 
        \mathrm{CM}(\mathcal{X}',\mathcal{L}'),
    \end{equation*}
    where the inequality is strict if and only if for all one-parameter subgroups $\beta$ of $T$, the twisted model $(\mathcal{X}_\beta,\mathcal{L}_\beta)$ is not isomorphic to $(\mathcal{X}',\mathcal{L}')$.
\end{theorem}
The rational one-parameter subgroup $\xi$ may be given explicitly in terms coefficients of Hilbert polynomials (see Remark \ref{rmk : xi F is explicit}).
Analytically, extremal metrics arise as critical points of the Calabi functional and as fixed points of the Calabi flow up to biholomorphisms, and are a natural generalization of cscK metrics when automorphisms are involved (see \cite[Section 4]{szekelyhidi2014introduction}).
Algebraically, their existence is conjectured to be equivalent to relative K-stability \cite{szekelyhidi2007extremal}, and they are relevant in the understanding of instability stratifications of the stack of polarized varieties.

In the special case of cscK metrics, one obtains the following:
\begin{corollary}\label{cor : main corollary}
        Let $(\mathcal{X},\mathcal{L})$ and $(\mathcal{X}',\mathcal{L}')$ be $T$-equivariant models of $(X,L)$ such that the restricted torus $T_0$ in $\mathrm{Aut}(\mathcal{X}_0, \mathcal{L}_0)$ is maximal. If the special fiber $(\mathcal{X}_0, \mathcal{L}_0)$ is smooth and admits a cscK metric, then
    \begin{equation*}
        \mathrm{CM}(\mathcal{X},\mathcal{L})
        \leq 
        \mathrm{CM}(\mathcal{X}',\mathcal{L}'),
    \end{equation*}
    where the inequality is strict if and only if for all one-parameter subgroups $\beta$ of $T$, the twisted model $(\mathcal{X}_\beta,\mathcal{L}_\beta)$ is not isomorphic to $(\mathcal{X}',\mathcal{L}')$.
\end{corollary}
Note that even in the case of discrete automorphisms, Corollary \ref{cor : main corollary} is a generalization of Theorem \ref{thm : hattoris theorem}, as we allow models over DVRs rather than over proper curves.
For this, we use a definition of the difference of CM degrees through the Knudsen--Mumford expansion, which we show recovers the definition used by Hattori when the models over $\Spec R$ are the restriction of models over proper curves.
Our proof in this generality then involves proving a generalization of the Wang--Odaka intersection-theoretic formula for the Donaldson--Futaki invariant involved in K-stability \cite{odaka2013generalization,wang2012height} to pairs of models over DVRs (Theorem \ref{thm : difference of CM degrees}).

Theorem \ref{thm : main theorem} proves separatedness of the moduli space of extremal manifolds with fixed maximal torus, a result that may also be obtained using Seyyedali's proof of relative asymptotic Chow stability of extremal manifolds \cite{seyyedali2017relative} (see also \cite{Mab18,ST21,Has21}); in fact, we will use this in order to prove Theorem \ref{thm : main theorem}.
Both Theorem \ref{thm : main theorem} and Corollary \ref{cor : main corollary} may be viewed as generalizations of results on relative K-(poly)stability of extremal and cscK manifolds due to Stoppa \cite{stoppa2009k} and Stoppa--Sz\'ekelyhidi \cite{stoppa2011relative} to families of polarized varieties.

We also obtain relative arc K-polystability of extremal manifolds:
\begin{corollary}\label{cor : arc corollary}
    Let $(Y,H)$ be an extremal polarized manifold and let $T$ be a maximal complex torus in $\Aut (Y,H)$. Then for all $T$-equivariant arcs $(\mathcal{Y},\mathcal{H})$ for $(Y,H)$, we have
    \begin{equation*}
        \mathrm{DF}_T(\mathcal{Y},\mathcal{H}) \geq 0,
    \end{equation*}
    where the inequality is strict if and only if for all one-parameter subgroups $\beta$ of $T$, the twisted arc $(\mathcal{Y}_\beta,\mathcal{H}_\beta)$ is not trivial.
\end{corollary}
This extends arc K-polystability of cscK manifolds to the extremal setting, proven by Dervan--Reboulet \cite[Corollary 1.2]{dervan2024arcs} (see also \cite{szekelyhidi2015filtrations,donaldson2012stability} in the absence of automorphisms).
The relative Donaldson--Futaki invariant for arcs used in Corollary \ref{cor : arc corollary} may be defined as the difference of CM degrees between a twist of the arc and the trivial arc, so that Corollary \ref{cor : arc corollary} follows directly from Theorem \ref{thm : main theorem}.

\subsection{Strategy of the proof}

We briefly outline the strategy of proving Theorem \ref{thm : main theorem}, which is divided into two steps:
the first step is proving that, given two equivariant models $(\mathcal{X},\mathcal{L})$ and $(\mathcal{X}',\mathcal{L}')$, there exists a torus-invariant good filtration $\mathcal{F}$ of $(\mathcal{X}_0,\mathcal{L}_0)$, such that
\begin{equation}\label{eq : step 1}
    \mathrm{DF}_{T_0} (\mathcal{F})
    = \mathrm{CM}(\mathcal{X}',\mathcal{L}')
    - \mathrm{CM}(\mathcal{X}_\xi,\mathcal{L}_\xi),
\end{equation}
for some rational one-parameter subgroup $\xi$ of $T$.
This filtration was constructed by Hattori \cite[Theorem 3.5]{hattori2024minimizing} in the non-equivariant case and generalizes a construction by Blum--Xu \cite{blum2019uniqueness} for degenerations of Fano varieties.
The left-hand side of Equation \eqref{eq : step 1} is the relative Donaldson--Futaki invariant for invariant good filtrations, generalizing Hattori's Donaldson--Futaki invariant for good filtrations, and the right-hand side is difference CM degrees of a pair of models over a DVR, both of which we will introduce in the present paper.

The second step is proving a new stability result involving the left-hand side of the Equation \eqref{eq : step 1} in the case where $(\mathcal{X}_0,\mathcal{L}_0)$ is smooth and extremal.
More precisely, we prove that for all torus-invariant good filtrations $\mathcal{F}$, we have $\mathrm{DF}_{T_0} (\mathcal{F}) \geq 0$, where the inequality is strict whenever the reduced norm $||\mathcal{F}||_{T_0}$ is strictly positive.
This generalizes works by Sz\'ekelyhidi \cite{szekelyhidi2015filtrations}, Hattori \cite{hattori2024minimizing} and Stoppa--Sz\'ekelyhidi \cite{stoppa2011relative} for extremal and cscK manifolds with automorphisms.

\subsection{Structure of the paper}\label{subsec : structure of the paper}

This paper is divided into three sections.
In the second section, we recall the necessary prerequisites on models, CM degrees and filtrations.
Notably, we introduce the notion of the difference of CM degrees of a pair of models over a DVR and prove a generalization of the Wang--Odaka intersection-theoretic formula in this setting. 

In the third section, we establish the second step of the strategy, proving that an extremal polarized manifold is relatively K-polystable with respect to good filtrations (Theorem \ref{thm : realtive good filtration K stab of extr}).
To prove this, we extend Sz\'ekelyhidi's blowup formula for Chow weights of filtrations \cite[Page 18]{szekelyhidi2015filtrations} to the relative setting, which involves producing an asymptotic expansion for inner products.
We also extend work of Boucksom \cite[Section 8]{szekelyhidi2015filtrations} on asymptotic vanishing orders of graded linear series to find a destabilizing invariant point to blow up.
Ultimately, we prove that extremal manifolds are relatively asymptotically Chow stable with respect to filtrations.

In the fourth and final section, we prove the first step of the strategy by giving an understanding of the twist of Hattori's good filtration, and complete the proofs of Theorem \ref{thm : main theorem} and Corollary \ref{cor : main corollary}.
Notably, we prove that the twist of the good filtration of two models coincides with the good filtration of the twist of one of the models.

\subsection{Acknowledgments}

I would like to thank my PhD supervisor Ruadha\'i Dervan for the continuous help throughout the journey of the present work and for many helpful discussions.
I would also like to thank Masafumi Hattori, R\'emi Reboulet, Lars Martin Sektnan, Xiaowei Wang and Mahmoud Abdelrazek for helpful conversations and comments on earlier versions of the present paper.
I received funding from a PhD studentship associated to Dervan's Royal Society University Research Fellowship (URF\textbackslash R1\textbackslash 201041).

\section{Preliminaries}

\subsection{The CM line bundle and models}\label{subsec : models and CM degrees}

We start with the introduction of the CM line bundle over a general irreducible base, following the treatment in \cite{FR06}.
Let $(\mathcal{X},\mathcal{L}) \to B$ be a flat family of $n$-dimensional $\Q$-polarized varieties over an irreducible scheme $B$.
This consists of a flat morphism $\mathcal{X} \to B$, of relative dimension $n$, and a relatively semi-ample $\Q$-line bundle $\mathcal{L}$ on $\mathcal{X}$. Assume that $n \geq 1$.

By the Knudsen--Mumford expansion \cite[Theorem 4]{KM76}, there exist $\Q$-line bundles $\lambda_{n+1}(\mathcal{X},\mathcal{L})$ and $\lambda_n(\mathcal{X},\mathcal{L})$ on $B$ such that
\begin{equation*}
    \det \pi_* (k\mathcal{L})
    = \frac{k^{n+1}}{(n+1)!}\lambda_{n+1}(\mathcal{X},\mathcal{L})
   - \frac{k^n}{2n!}\lambda_n(\mathcal{X},\mathcal{L})
    + O(k^{n-1})
\end{equation*}
as $k \to +\infty$, for sufficiently divisible $k$.
Moreover, by flatness, the coefficients of the Hilbert polynomial $h^0(\mathcal{X}_b,k\mathcal{L}_b) = a_0 k^n + a_1 k^{n-1} + O(k^{n-2})$ are independent of choice of $b$ in $B$.
\begin{definition}
    The \textit{Chow--Mumford (CM) line bundle} associated to the family $(\mathcal{X},\mathcal{L}) \to B$ is the $\Q$-line bundle
    \begin{equation*}
        \lambda_\mathrm{CM}(\mathcal{X},\mathcal{L})
        = \frac{1}{2V(\mathcal{X}_b,\mathcal{L}_b)} \bigg(
        \lambda_n(\mathcal{X},\mathcal{L}) 
        -
        \frac{n\mu(\mathcal{X}_b,\mathcal{L}_b)}{n+1} \lambda_{n+1}(\mathcal{X},\mathcal{L})
        \bigg),
    \end{equation*}
    where $V(\mathcal{X}_b,\mathcal{L}_b) = n!a_0$ and $\mu(\mathcal{X}_b,\mathcal{L}_b) = -\frac{2a_1}{na_0}$ denote the volume and slope of any fiber $(\mathcal{X}_b,\mathcal{L}_b)$.
\end{definition}
The definition above is due to Paul--Tian \cite{PT06}, while the original construction is due to Fujiki--Schumacher \cite{FS90}, where it was used to study families and construct moduli of polarized varieties.
In the Fano case, the CM line bundle is known to be ample on the good moduli space of K-polystable Fano varieties, proving its projectivity \cite{xu2020positivity,codogni2021positivity}.

Our definition of the CM line bundle is such that it is invariant under scaling, that is,
\begin{equation*}
    \lambda_\mathrm{CM}(\mathcal{X},r\mathcal{L})
    = \lambda_\mathrm{CM}(\mathcal{X},\mathcal{L}),
\end{equation*}
for all integers $r \geq 1$.
In the case where the structure morphism $\pi : \mathcal{X} \to B$ admits a relative canonical bundle $K_{\mathcal{X}/B}$, by Grothendiek--Riemann--Roch, the first Chern class of the CM line bundle satisfies
\begin{equation*}
    c_1(\lambda_\mathrm{CM}(\mathcal{X},\mathcal{L}))
    = \frac{\pi_* \big( c_1(\mathcal{L})^n c_1(K_{\mathcal{X}/B}) - \frac{n\mu(\mathcal{X}_b,\mathcal{L}_b)}{n+1} c_1(\mathcal{L})^{n+1} \big)}{2V(\mathcal{X}_b,\mathcal{L}_b)} .
\end{equation*}

\subsubsection{Models over DVRs}\label{subsubsec : models and dvrs}

We next discuss models, which will be the class of families of interest in the present paper.
Let $R$ be a discrete valuation ring with fraction field $K$ and residue field $\C$, and choose a valuation $\mathsf{v} : K^* \to \R$. Let $(X,L)$ be an $n$-dimensional normal $\Q$-polarized $K$-variety.
This means that $X \to \Spec K$ is a flat morphism of relative dimension $n$, and $L$ is a relatively ample $\Q$-line bundle on $X$.
\begin{definition}\label{def : model}
    An \textit{$R$-model} of $(X,L)$ is a flat normal projective relatively $\Q$-polarized integral scheme $(\mathcal{X},\mathcal{L})$ over $\Spec R$, together with an isomorphism $(\mathcal{X}_K,\mathcal{L}_K) \simeq (X,L)$ over $\Spec K$, where $(\mathcal{X}_K,\mathcal{L}_K)$ is the base change of $(\mathcal{X},\mathcal{L})$ over the generic point.
    Two models $(\mathcal{X},\mathcal{L})$ and $(\mathcal{X}',\mathcal{L}')$ are said to be \textit{isomorphic} if there exists an isomorphism over $\Spec R$ between them that extends the composition of isomorphisms $(\mathcal{X}_K,\mathcal{L}_K) \simeq (X,L) \simeq (\mathcal{X}_K',\mathcal{L}_K')$.
\end{definition}

The CM line bundles $\lambda_\mathrm{CM}$ and $\lambda'_\mathrm{CM}$ of two models $(\mathcal{X},\mathcal{L})$ and $(\mathcal{X}', \mathcal{L}')$ are canonically isomorphic over the generic fiber, and using the identification with $(X,L)$, one may identify the difference of CM line bundles as a rational number, as follows:
by functoriality of the Knudsen--Mumford expansion, the identification $(\mathcal{X}_K,\mathcal{L}_K) \simeq (X,L) \simeq (\mathcal{X}_K',\mathcal{L}_K')$ gives rise to an explicit isomorphism
\begin{equation*}
    \Phi : \lambda_\mathrm{CM}|_{\Spec K}
    \to \lambda_\mathrm{CM}'|_{\Spec K}
\end{equation*}
of line bundles. This isomorphism may also be viewed as an isomorphism of $K$-vector spaces
\begin{equation*}
    \Phi_* : H^0 (\Spec K, m  \lambda_\mathrm{CM} )
    \to H^0 (\Spec K, m  \lambda_\mathrm{CM}' ),
\end{equation*}
where $m$ is any sufficiently divisible integer such that one takes sections of line bundles.
The domain and target contain $\Lambda = H^0(\Spec R, m  \lambda_\mathrm{CM})$ and $\Lambda' = H^0(\Spec R, m  \lambda_\mathrm{CM}')$ as $R$-submodules, respectively.
\begin{definition}\label{def : difference of CM degrees}
    The \textit{difference of CM degrees} of two models $(\mathcal{X},\mathcal{L})$ and $(\mathcal{X}', \mathcal{L}')$ of $(X,L)$ is the valuation
    \begin{equation*}
        \mathrm{CM}(\mathcal{X}, \mathcal{L})
        - \mathrm{CM}(\mathcal{X}',\mathcal{L}')
        := m^{-1}\mathsf{v}(\gamma),
    \end{equation*}
    where $\gamma$ is any non-zero element $K$ such that $\Phi_*(\Lambda) = \gamma \cdot \Lambda'$.
    We write $\mathrm{CM}(\mathcal{X}, \mathcal{L}) \geq \mathrm{CM}(\mathcal{X}',\mathcal{L}')$, whenever $\mathrm{CM}(\mathcal{X}, \mathcal{L}) - \mathrm{CM}(\mathcal{X}',\mathcal{L}') \geq 0$, and similarly for other inequalities.
\end{definition}
Within K-stability, the idea of associating a number to two line bundles over the formal disk endowed with a fixed identification away from the special fiber goes back to Donaldson \cite{donaldson2012stability} (see also Wang \cite{wang2012height}).
In the case $R = \C [\![ t ]\!]$, the number $\mathsf{v}(\gamma)$ coincides with the unique largest integer $\nu$ such that, giving a trivializing section $s$ of $\Lambda$, the section $t^{-\nu}\Phi_*s$ extends to a section on $\Spec \C [\![ t ]\!]$.
In this case, the difference of CM degrees is equal to the degree of the line bundle $\overline{\lambda_\mathrm{CM} - \lambda_\mathrm{CM}'}$ on $\P^1$ obtained by gluing the line bundle $\lambda_\mathrm{CM} - \lambda_\mathrm{CM}'$ on $\Spec \C [\![ t ]\!]$ to the trivial line bundle on $\P^1 \setminus 0$, using the explicit identification $\lambda_\mathrm{CM} - \lambda_\mathrm{CM}'|_{\Spec \C (\!( t )\!)} \simeq \O_{\Spec \C (\!( t )\!)}$ on the intersection.

One may describe the difference of CM degrees in terms of coefficients of Hilbert polynomials associated to the models:
the composition of isomorphisms $\mathcal{X}'_K \simeq X \simeq \mathcal{X}_K$ gives rise to a birational map $\mathcal{X}' \dashrightarrow \mathcal{X}$. Choose a resolution of indeterminacies
\begin{equation}\label{eq : res of indeterm}
\begin{tikzcd}[row sep = 0.1em]
    & \mathcal{Y} \ar[dddddddddddddddddl,swap,"\mu'"] \ar[dddddddddddddddddr,"\mu"] &
    \\\\\\\\\\\\\\\\
    \\\\\\\\\\\\\\\\\\
    \mathcal{X}' \ar[rr,dashed] && \mathcal{X},
\end{tikzcd}
\end{equation}
compatible with the morphisms to $\Spec R$. 
After potentially changing $(\X,\L)$ with $(\X',\L')$, we can write $\mu^*\mathcal{L} = \mu'^* \mathcal{L}' + E - a\mathcal{Y}_0$ for an effective $\Q$-divisor $E$ supported in $\mathcal{Y}_0$ and a rational number $a > 0$.
This gives rise to an embedding
\begin{equation*}
    H^0(\mathcal{Y},k\mathcal{M}) \subset H^0(\mathcal{X},k\mathcal{L}),
\end{equation*}
for $\mathcal{M} = \mu'^*\mathcal{L}' - a\mathcal{Y}_0$ and sufficiently divisible $k$.

Assume that $-E$ is $\mu$-relatively ample by replacing $(\mathcal{Y},\mathcal{M})$ with the ample model $(\Proj_{\mathcal{X}}(\bigoplus_{k \geq 0} \mu_*\cO_\mathcal{Y}(k\mathcal{L})),\cO(1))$ over $(\mathcal{X},\mathcal{L})$.
There exists a closed subscheme $\mathcal{Z}$ in $\mathcal{X}$ supported in the special fiber such that the morphism $\mu : \mathcal{Y} \to \mathcal{X}$ is the blow-up of $\mathcal{X}$ along $\mathcal{Z}$ with exceptional divisor $rE$ for an integer $r \geq 1$.
By \cite[Lemma 5.4.24]{lazarsfeld2017positivity}, we obtain
\begin{equation*}
    H^i(\mathcal{Y},k\mathcal{M})
    \simeq  H^i(\mathcal{X},kr\mathcal{L} \otimes \mathcal{I}_\mathcal{Z}^{\otimes k}),
\end{equation*}
for all $i \geq 0$ and sufficiently large $k$. In particular, the dimension of the quotient $H^0(\mathcal{X},k\mathcal{L}) / H^0(\mathcal{Y},k\mathcal{M})$ is finite and computed over the special fiber.
\begin{theorem}\label{thm : difference of CM degrees}
    There exist rational numbers $b_0$ and $b_1$, such that
    \begin{equation}\label{eq : expansion of dim of quotient}
        \dim_\C \big(H^0(\mathcal{X},k\mathcal{L}) / H^0(\mathcal{Y},k\mathcal{M}) \big)
        = b_0 k^{n+1}
        + b_1 k^n
        + O(k^{n-1})
    \end{equation}
    as $k \to +\infty$, for sufficiently divisible $k$. Moreover
    \begin{equation*}
        \mathrm{CM}(\mathcal{X}, \mathcal{L})
        - \mathrm{CM}(\mathcal{X}',\mathcal{L}')
        = \frac{b_0 a_1 - b_1 a_0}{a_0^2},
    \end{equation*}
    where $h^0(\mathcal{X}_0,k\mathcal{L}_0)
    = h^0(\mathcal{X}_0',k\mathcal{L}_0')
    = a_0 k^n
    + a_1 k^{n-1}
    + O(k^{n-2})$.
\end{theorem}

In the case where the models are restrictions of flat families $(\overline{\mathcal{X}},\overline{\mathcal{L}}) \to C$ and $(\overline{\mathcal{X'}},\overline{\mathcal{L'}}) \to C$ for a smooth projective curve $C$, Hattori proved
\begin{equation*}
    \dim \!\big(H^0(\mathcal{X},k\mathcal{L}) / H^0(\mathcal{Y},k\mathcal{M}) \big)
    = \chi(\overline{\mathcal{X}},k\overline{\mathcal{L}})
    - \chi(\overline{\mathcal{X}'},k(\overline{\mathcal{L}'} + a\mathcal{X}_0'))
    + O(k^{n-1}),
\end{equation*}
\cite[Page 19]{hattori2024minimizing}. It follows that Theorem \ref{thm : difference of CM degrees} recovers Hattori's definition of the difference of CM degrees of two models over a projective curve.

\begin{proof}
    We may assume that $R$ is regular by normalizing the base, pulling back the family and choosing an ample model; as the Knudsen--Mumford expansion is functorial, the difference of CM degrees is unchanged, while Equation \eqref{eq : expansion of dim of quotient} is computed over the central fiber. 
    
    Write $E_k = \pi_* (k\L)$ as a vector bundle on $\Spec R$, where $\pi : \mathcal{X} \to \Spec R$ is the structure morphism, as write
    \begin{align*}
        H_k^i &= H^i(\Spec K, E_k),
        \\
        H_{k,R}^i &= H^i(\Spec R, E_k).
    \end{align*}
    Similarly, write $E_k'$, $H_k'^{\,i}$ and $H_{k,R}'^{\,i}$ for the same objects, but using the model $(\X',\L')$.
    Similar to Definition \ref{def : difference of CM degrees}, one can define the degree of the vector bundle $E_k \otimes E_k'^{\,\vee}$ with respect to the identification $E_k|_{\Spec K} \simeq E'_k|_{\Spec K}$, explicitly given by the models, as the valuation
    \begin{equation*}
        \deg \big( E_k \otimes E_k'^{\,\vee} \big) := \mathsf{v}(\det \gamma),
    \end{equation*}
    where $\gamma$ is any element in $\GL_K(H_k'^{\,0})$ such that $\Phi(H_{k,R}^0) = \gamma(H_{k,R}'^{\,0})$, where $\Phi : H_k^i \to H_k'^{\,i}$ is the isomorphism associated to the identification.
    Note that by construction, we have
    \begin{equation}\label{eq : deg det E_k = deg E_k}
        \deg (\det E_k - \det E_k')
        = \deg ( E_k \otimes E_k'^{\,\vee} ),
    \end{equation}
    where the degree of the difference of the determinant line bundles is defined the same way as for the difference of CM line bundles.

    The localized Euler characteristic of $E_k \otimes E_k'^{\,\vee}$ with respect to the identification is defined as the difference of lengths or $R$-modules
    \begin{align*}
        \chi \big( E_k \otimes E_k'^{\,\vee} \big)
        &= \sum_{i \geq 0} (-1)^i \, \mathrm{len}_R \big(H_k^i/(H_{k,R}^i \cap \Phi^{-1}(H_{k,R}'^{\,i}) \big).
    \end{align*}
    Note that $H^i_k = 0$ for all $i > 0$ and $k$ very large, as $\Spec R$ is affine, and that $\Phi^{-1}(H_{k,R}'^{\,0}) = H^0(\Spec R,\pi_*^\mathcal{Y}(k\mathcal{M})) =: H^\mathcal{Y}_{k,R}$, by construction of $\mathcal{Y}$, where $\pi^\mathcal{Y} : \Y \to \Spec R$ is the structure morphism. 
    
    By the localized Riemann--Roch formula over $\Spec R$ (see for example \cite[Theorem 18.2]{fulton2013intersection}, where we apply the theory to the closed subscheme $0$ in $\Spec R$), we then have
    \begin{equation*}
    \begin{aligned}
        \deg \big( E_k \otimes E_k'^{\,\vee} \big)
        &= \chi \big( \pi_*(k\L) \otimes \pi'_*(k\L')^\vee \big),
        \\
        &= \mathrm{len}_R \big( H_k^0 / H^\mathcal{Y}_{k,R} \big),
        \\
        &= \dim_\C \big( H_{k,R}^0 / H^\mathcal{Y}_{k,R} \big),
        \\
        &= \dim_\C \big(H^0(\mathcal{X},k\mathcal{L}) / H^0(\mathcal{Y},k\mathcal{M}) \big),
    \end{aligned}
    \end{equation*}
    for $k$ very large, where we used \cite[Proposition 8.5]{hartshorne2013algebraic} in the last line.
    Applying the Knudsen--Mumford expansion to Equation \eqref{eq : deg det E_k = deg E_k} and taking the degree, we obtain
    \begin{equation}\label{eq : coeff b_0 for models}
    \begin{aligned}
        b_0
        &= \frac{\deg ( \lambda_{n+1}(\mathcal{X},\mathcal{L})
        -  \lambda_{n+1}(\mathcal{X}',\mathcal{L}'))}{(n+1)!},
        \\
        b_1
        &= -\frac{\deg (\lambda_n(\mathcal{X},\mathcal{L})
        - \lambda_n(\mathcal{X}',\mathcal{L}'))}{2n!}.
    \end{aligned}
    \end{equation}
    This yields the result.
\end{proof}

\begin{remark}\label{rmk : intersection formula for diff of DM deg}
    With further hypotheses on the family $\mathcal{Y} \to \Spec R$, the difference of CM degrees may be defined in different ways.
    In the case where the models are normal and compactify to a family over a smooth proper curve $C$, we have
    \begin{equation*}
        \mathrm{CM}(\mathcal{X}, \mathcal{L})
        - \mathrm{CM}(\mathcal{X}',\mathcal{L}')
        = \frac{(\mu^*\mathcal{L}^n - \mathcal{L}'^{\,n}).K_{\mathcal{X}/C} - \frac{n\mu}{n+1}(\mu^*\mathcal{L}^{n+1} - \mathcal{L}'^{\,n+1})}{2V}.
    \end{equation*}
    If the models are smooth and the family are restrictions of famililies over a smooth (not necessarily proper) curve, the difference of CM degrees may also be viewed through Deligne pairings, as in this case the $\Q$-line bundles $\lambda_i$ in the Knudsen-Mumford expansion may be interpreted through Deligne pairings \cite{PRS08}.
    Theorem \ref{thm : difference of CM degrees} generalizes the formula for the Donaldson--Futaki invariant for test configurations and the difference of CM degrees for families over proper curves due to Odaka \cite{odaka2013generalization} and Wang \cite{wang2012height}.
\end{remark}

\begin{example}\label{exp : models generalize test configurations}
    Within K-stability, examples of models are ones induced by test configurations \cite[Definition 2.1.1]{donaldson2002scalar} (see also \cite[Example 2.2]{dervan2024arcs}). A test configuration is a $\C^*$-equivariant degeneration $(\mathcal{Y},\mathcal{H}) \to \C$ of a polarized variety $(Y,H)$ and its Donaldson--Futaki invariant is the difference of CM degrees
    \begin{equation*}
        \mathrm{DF}(\mathcal{Y},\mathcal{H})
        = \mathrm{CM}(\mathcal{Y}_{\C[\![t]\!]},\mathcal{H}_{\C[\![t]\!]})
        - \mathrm{CM}(Y \times \Spec \C[\![t]\!],H),
    \end{equation*}
    between the restriction to the formal disk $\Spec \C[\![t]\!]$ and the trivial model.
\end{example}

\subsubsection{Equivariant and twisted models}\label{subsubsec: twisted and equivariant models}

We introduce a subclass of models that are compatible with vertical torus actions on $(X,L) \to \Spec K$.
Let $T$ be a $K$-torus in $\mathrm{Aut}_K(X,L)$, meaning that $T \cong \G_{m,K}^d$ for some integer $d \geq 0$.
\begin{definition}
    A model $(\mathcal{X},\mathcal{L})$ of $(X,L)$ is called \textit{$T$-equivariant} if it admits an $R$-torus $T_R$ in $\mathrm{Aut}_R(\mathcal{X},\mathcal{L})$ such that $T_R|_{(\mathcal{X}_K,\mathcal{L}_K)} = \varphi^*T$, where $\varphi : (\mathcal{X}_K,\mathcal{L}_K) \to (X,L)$ is the designated identification associated to the model $(\mathcal{X},\mathcal{L})$.
\end{definition}
This is a generalization of invariant test configurations due to Sz\'ekelyhidi \cite[Definition 2.1]{szekelyhidi2007extremal}.
As the generic fiber $\mathcal{X}_K$ is dense in $\mathcal{X}$, if such a torus $T_R$ exists, it must be unique.

Choose a uniformizer $t$ in the discrete valuation ring, that is, an element in $R$ such that its valuation $\mathsf{v}(t)$ generates the valuation group $\mathsf{v}(K^*)$ in $\R$; (in the case $R = \C[\![t]\!]$, the element $t$ itself is a uniformizer).
Let $(\mathcal{X},\mathcal{L})$ be a model of $(X,L)$ and let $\beta$ be a vertical $\G_{m,K}$-action on $(X,L)$. 
The choice of uniformizer $t$ gives rise to an automorphism $\beta(t) : (X,L) \to (X,L)$.
\begin{definition}\label{def : twisted model and equivariant model}
    The \textit{twisted model} of $(\mathcal{X},\mathcal{L})$ by $\beta$ is defined to be the model $(\mathcal{X}_\beta,\mathcal{L}_\beta) := (\mathcal{X},\mathcal{L})$ whose identification with $(X,L)$ has been replaced by $\varphi_\beta := \beta(t) \circ \varphi$.
\end{definition}
The definition of twisted models extends the notion of twisted test configurations introduced by Sz\'ekelyhidi \cite{szekelyhidi2007extremal}, and further developed by Li \cite{li2022g}.
Note that, unlike test configurations, models do not need to be equivariant in order for them to be twisted.

One may also twist by rational one-parameter subgroups $\xi$ of $T$; that is, elements of the vector space $\Hom (\G_{m,K},\Aut_K(X,L)) \otimes \Q$.
If $\xi = r^{-1}\beta$ for an integer $r \geq 1$ and a one-parameter subgroup $\beta$ of $T$, then we define the twist of $(\mathcal{X},\mathcal{L})$ by $\xi$ as the twist
\begin{equation*}
    (\mathcal{X}_\xi, \mathcal{L}_\xi)
    := (\mathcal{X}, r^{-1} \mathcal{L}_{r\beta}).
\end{equation*}
This means that the identification over $\Spec K$ takes the form
\begin{equation*}
    (\mathcal{X}_{\xi,K}, \mathcal{L}_{\xi,K})
    = (\mathcal{X},r^{-1}\mathcal{L})
    \simeq (X,r^{-1}L)
    \xrightarrow{r\beta(t)}
    (X,L).
\end{equation*}

\begin{remark}  
    In the more geometrical case of families $(X,L) \xrightarrow{\pi} \Delta^*$ over the punctured disk in $\C$, one would choose a vertical $\C^*$-action $\beta$ on $(X,L)$, giving rise to an automorphism $(X,L)$, sending a point $(x,\xi) \mapsto \beta(\pi(x))(x,\xi)$ for $x$ in $X$ and $\xi$ in $L_x$.
    The twist of a filling $(\mathcal{X},\mathcal{L}) \to \Delta$ to the whole disc would be defined in the same way by composing the identification with the automorphism.
\end{remark}

\subsection{Invariant, twisted and good filtrations.}

In this section we recall the notions of filtrations and Okounkov bodies of complex polarized varieties, following the treatment in \cite{szekelyhidi2015filtrations} (up to a change in sign) and give the definitions of invariant, twisted and good filtrations, introduced by Li \cite{li2022g} and Hattori \cite{hattori2024minimizing}.

Let $(Y,H)$ be an $n$-dimensional complex polarized variety with ring of sections $S = \bigoplus_{k \geq 0} S_k$, where $S_k = H^0(Y,kH)$.
\begin{definition}
    A \textit{(unitary, decreasing, multiplicative, linearly bounded}) $\Z$\textit{-filtration} of $(Y,H)$ is a collection $\mathcal{F} = (\mathcal{F}^\lambda S_k)_{\lambda,k}$ of linear subspaces $\mathcal{F}^\lambda S_k$ in $S_k$ for $\lambda$ in $\Z$ and $k$ in $\Z_{\geq 0}$ such that for all integers $\lambda \leq \mu$ and and all integers $k,l \geq 0$, we have
\begin{enumerate}[(i)]
    \item $1 \in \mathcal{F}^\lambda S_0$,
    \item $\mathcal{F}^\lambda S_k \supset \mathcal{F}^\mu S_k$,
    \item $\mathcal{F}^\lambda S_k \cdot \mathcal{F}^\mu S_l \subset \mathcal{F}^{\lambda + \mu} S_{k + l}$,
    \item $\mathcal{F}^\lambda S_k = S_k$ for all $\lambda < -Ck$
    and $\mathcal{F}^\lambda S_k = 0$ for all $\lambda \geq -Ck$,
\end{enumerate}
for some constant $C > 0$, independent from $\lambda$ and $k$.
The \textit{$p$\textsuperscript{th}-order weight function} of $\mathcal{F}$ is defined to be the function $w_\mathcal{F}^p : \Z \to \Z$, given by
\begin{equation*}
    w_\mathcal{F}^p(k)
    = \sum_{\lambda \in \Z} \lambda^p \dim (\mathcal{F}^\lambda S_k / \mathcal{F}^{\lambda + 1} S_k).
\end{equation*}
We write $w_\mathcal{F} = w_\mathcal{F}^1$.
The \textit{$L^2$-norm} of $\mathcal{F}$ is given by
\begin{equation*}
    ||\mathcal{F}||^2
    = \frac{c_0 a_0 - b_0^2}{a_0^2},
\end{equation*}
where $a_0$, $b_0$ and $c_0$ are real numbers coming from the expansions
\begin{equation}\label{eq : expansion for a_0 b_0 c_0}
\begin{aligned}
    \dim S_k
    &= a_0 k^n + O(k^{n-1}),
    \\
    w_\mathcal{F}(k)
    &= b_0 k^{n+1} + o(k^{n+1}),
    \\
    w_\mathcal{F}^2(k)
    &= c_0 k^{n+2} + o(k^{n+2}),
\end{aligned}
\end{equation}
as $k \to +\infty$.
We call $\mathcal{F}$ \textit{finitely generated} if the $\C[t]$-algebra $\bigoplus_{k,\lambda} \mathcal{F}^\lambda S_k \cdot t^{-\lambda}$ is finitely generated.
\end{definition}
The numbers $b_0$ and $c_0$ exist as the filtration is linearly bounded.
\begin{example}[{\cite[Definition 6.2]{nystrom2012test}}]\label{exp : filtrations and test configurations}
    An important example of filtrations are the ones induced by test configurations.
    In the case of a product test configuration $(Y \times \C,H)_\beta$ induced by a $\C^*$-action $\beta$ on $(Y,H)$, the associated filtration $\mathcal{F}_{(Y \times \C,H)_\beta}$ takes the form
    \begin{equation*}
        \mathcal{F}^\lambda_{(Y \times \C,H)_\beta} S_k
        = \bigoplus_{\mu \geq \lambda} (S_k)_\mu,
    \end{equation*}
    where $S_k = \bigoplus_{\mu \in \Z} (S_k)_\mu$ denotes the weight-space decomposition with respect to the $\C^*$-action induced by $\beta$.
    In particular, anything that is defined for filtrations is also defined for $\C^*$-actions through the associated filtration.
\end{example}
\begin{example}[{\cite[Section 3.2]{szekelyhidi2015filtrations}}]\label{exp : approx filtrations}
    Any filtration $\mathcal{F}$ may be approximated by finitely generated filtrations $\mathcal{F}_{(r)}$ of $(Y,rH)$ for $r \geq 1$, given by
    \begin{equation*}\label{eq : tautological approximation}
        \mathcal{F}^\lambda_{(r)} S_{kr} = \{ \sigma \;|\; \sigma \cdot t^{-\lambda} \in \mathcal{S}_{kr}(\mathcal{F}) \},
    \end{equation*}
    where $\mathcal{S}_k(\mathcal{F}) \subset S[t]$ is the $\C[t]$-subalgebra generated by $\bigoplus_\lambda \mathcal{F}^\lambda S_k \cdot t^{-\lambda}$.
    We call the sequence of filtrations $(\mathcal{F}_{(r)})$ the \textit{tautological approximation} of $\mathcal{F}$.
\end{example}

\subsubsection{Okounkov bodies and filtrations.}\label{subsubsec : okounkov bodies and filtrations}

Okounkov bodies \cite{okounkov1996brunn} play an important role in capturing asymptotic data of ample line bundles and filtrations.
For this section, assume that $(Y,H)$ is smooth.

Given local coordinates $(U,z)$ centered around a point $p$ in $Y$ and a global section $\sigma$ of $L$ which does not vanish at $p$, we define a convex subset
\begin{equation}\label{eq : Okounkov body}
    P = \overline{\bigcup_{k \geq 1} \frac{P_k}{k}}
    \subset \R^n,
\end{equation}
where $P_k = \{ \nu_\sigma(f) \;|\; f \in S_k \} \subset \Z^n$ and $\nu_\sigma(f)$ is the multi-degree of the lowest-order term of the power series $q$ in $\Z [\![ z_1,\dots,z_n ]\!]$ such that $f = \sigma^k q$ on $U$ with respect to the lexicographic order.
We call $P$ the \textit{Okounkov body} of $(Y,H)$ centered around $p$, induced by the data $(U,z,\sigma)$.
Important properties are $|P_k| = \dim S_k$ and $\mathrm{vol}(P) = n!\mathrm{vol}(L)$.
From \cite[Lemma 6]{szekelyhidi2015filtrations}, if follows that for sufficiently large integers $k$ and $r$, we have $\Delta_{k - 1} \cap \Z^n \subset P_{kr}$, where $\Delta_l = \{ (x_1, \dots, x_n) \in \R^n \;|\; x_i \geq 0 \text{ and } \sum_{i = 1}^n x_i \leq l \}$ denotes the closed solid $n$-simplex of side length $l > 0$ .
In particular, we have $\Delta_{1/r} \subset P$.

Let $\mathcal{F}$ be a filtration of $(Y,H)$.
\begin{definition}[{\cite[Section 3]{nystrom2012test}, \cite[Equation 13]{szekelyhidi2015filtrations}}]\label{def : g_Fk and G_F}
    The \textit{$k$\textsuperscript{th}-lattice function and the concave transform} of $\mathcal{F}$ with respect to $P$ are the functions $g_{\mathcal{F},k} : P_k \to \Z$ and $G_\mathcal{F} : P \to \R \cup \{ - \infty \}$, given by
    \begin{equation*}
    \begin{aligned}
        g_{\mathcal{F},k}(a) &= \max \big\{ \lambda \in \Z \;\big|\; a \in P_k\big(\tfrac{\lambda}{k},\mathcal{F}\big) \big\},
        \\
        G_\mathcal{F}(x) &= \sup \{s \in \R \;|\; x \in P(s,\mathcal{F}) \},
    \end{aligned}
    \end{equation*}
    where $P_k(s,\mathcal{F}) = \nu_\sigma(\mathcal{F}^{\lceil s k \rceil} S_k) \subset P_k$ and $P(s,\mathcal{F}) = \overline{\bigcup_{k \geq 1} \frac{P_k(s,\mathcal{F})}{k}} \subset P$.
\end{definition}
The function $G_\mathcal{F}$ is concave and upper semi-continuous on $P$ and is finite and continuous on the interior of $P$.
A key property (see \cite{nystrom2012test,szekelyhidi2015filtrations}) is that one may express the leading coefficient of the $p$\textsuperscript{th}-order weight function of $\mathcal{F}$ as
\begin{equation}\label{eq : weight function as sum of gk functions}
\begin{aligned}
    w_\mathcal{F}^p(k)
    &= \Big( \int_P G_\mathcal{F}^p \Big) k^{n+p} + o(k^{n+p}),
\end{aligned}
\end{equation}
as $k \to +\infty$, integrating with respect to the Lebesgue measure (as we do throughout our work). In particular, the $L^2$-norm of $\mathcal{F}$ may be written as
\begin{equation*}
    ||\mathcal{F}||^2
    = \frac{1}{\mathrm{vol}(P)} \int_P (G_\mathcal{F} - \overline{G}_\mathcal{F})^2,
\end{equation*}
where $\overline{G}_\mathcal{F} = \frac{1}{\mathrm{vol}(P)}\int_P G_\mathcal{F}$ denotes the average of $G_\mathcal{F}$ over $P$.

\subsubsection{Invariant and twisted filtrations.}

We introduce a subclass of filtrations which serves as a generalization of invariant test configurations \cite[Definition 2.1]{szekelyhidi2007extremal}.
Let $\beta$ be a $\C^*$-action on $(Y,H)$.
\begin{definition}[{\cite[Section 3.2]{li2022g}}]
    A filtration $\mathcal{F}$ of $(Y,H)$ is called \textit{$\beta$-invariant} if each of its filtered pieces $\mathcal{F}^\lambda S_k$ are invariant under the $\C^*$-action on $S_k$ induced by $\beta$.
\end{definition}
A $\beta$-invariant filtration $\mathcal{F}$ gives rise to weight-space decompositions
\begin{equation*}
\begin{aligned}
    \mathcal{F}^\lambda S_k &= \bigoplus_{\mu \in \Z} \big( \mathcal{F}^\lambda S_k \big)_\mu,
    \\
    \mathcal{F}^\lambda S_k / \mathcal{F}^{\lambda+1} S_k &= \bigoplus_{\mu \in \Z} \big( \mathcal{F}^\lambda S_k / \mathcal{F}^{\lambda+1} S_k \big)_\mu,
\end{aligned}
\end{equation*}
admitting natural identifications
$(\mathcal{F}^\lambda S_k)_\mu / (\mathcal{F}^{\lambda+1} S_k)_\mu \simeq (\mathcal{F}^\lambda S_k / \mathcal{F}^{\lambda+1} S_k)_\mu$.
\begin{definition}
    The \textit{mixed square-weight function} of a $\beta$-invariant filtration $\mathcal{F}$ and $\beta$ is the function $w_{\mathcal{F},\beta}^2 : \Z \to \Z$, given by
    \begin{equation*}
        w^2_{\mathcal{F},\beta}(k)
        = \sum_{\lambda,\mu \in \Z} \lambda \mu \dim \big(( \mathcal{F}^\lambda S_k / \mathcal{F}^{\lambda+1} S_k )_\mu \big),
    \end{equation*}
    The \textit{$L^2$-inner product} of $\mathcal{F}$ and $\beta$ is defined as
    \begin{equation*}
        \langle \mathcal{F}, \beta \rangle
        = \frac{c_0' a_0 - b_0 b_0'}{a_0^2},
    \end{equation*}
    where $a_0$ and $b_0$ are defined as in Equation \eqref{eq : expansion for a_0 b_0 c_0}, and, as $k \to +\infty$,
    \begin{equation*}
    \begin{aligned}
        w_\beta(k)
        &= b_0' k^{n+1} + O(k^n),
        \\
        w^2_{\mathcal{F},\beta}(k)
        &= c_0' k^{n+2} + o(k^{n+2}).
    \end{aligned}
    \end{equation*}
\end{definition}
The mixed square-weight and the inner product naturally extend the usual definitions for commuting $\C^*$-actions on $(Y,H)$ in \cite{szekelyhidi2007extremal}.
\begin{definition}[{\cite[Definition 3.6]{li2022g}}]
    The \textit{twisted filtration} of a $\beta$-invariant filtration $\mathcal{F}$ by $\beta$ is the filtration $\mathcal{F}_\beta$ of $(Y,H)$, defined by
    \begin{equation*}
        \mathcal{F}_\beta^\lambda S_k
        = \bigoplus_{\lambda_1 + \lambda_2 = \lambda} \big(\mathcal{F}^{\lambda_1} S_k \big)_{\lambda_2}.
    \end{equation*}
\end{definition}
The weight pieces of the twisted filtration are $(\mathcal{F}_\beta^\lambda S_k)_\mu = (\mathcal{F}^{\lambda - \mu} S_k)_\mu$.
\begin{lemma}\label{lem : weight function of twisted filtration}
    The first- and second-order weight functions of $\mathcal{F}_\beta$ satisfy
    \begin{equation*}
    \begin{aligned}
        w_{\mathcal{F}_\beta}
        &= w_\mathcal{F} + w_\beta,
        \\
        w_{\mathcal{F}_\beta}^2
        &= w_\mathcal{F}^2 + 2w^2_{\mathcal{F},\beta} + w^2_\beta.
    \end{aligned}
    \end{equation*}
    In particular, we have
    \begin{equation*}
        ||\mathcal{F}_\beta||^2
        = ||\mathcal{F}||^2
        + 2\langle \mathcal{F},\beta \rangle
        + ||\beta||^2.
    \end{equation*}
\end{lemma}
In the case of test configurations, this is proven using that the infinitesimal generator of a twisted test configuration is the sum of infinitesimal generators of the test configuration and the vertical $\C^*$-action.
In the proof for filtrations, this is hidden behind the definition of a twisted filtration which is chosen in such a way that it matches the definition of a twisted test configuration.
\begin{proof}
    Writing $S_k = \bigoplus_{\mu \in \Z} (S_k)_\mu$ for the weight-space decomposition with respect to the $\beta$-action, we compute directly:
    \begin{align*}
        w_{\mathcal{F}_\beta}(k)
        &= \sum_{\lambda_1,\lambda_2 \in \Z} (\lambda_1 + \lambda_2) \dim \big( ( \mathcal{F}^{\lambda_1}S_k / \mathcal{F}^{\lambda_1 + 1}S_k)_{\lambda_2} \big),
        \\
        &= \sum_{\lambda_1 \in \Z} \lambda_1 \dim \Big( \bigoplus_{\lambda_2 \in \Z} \big( \mathcal{F}^{\lambda_1}S_k / \mathcal{F}^{\lambda_1 + 1}S_k \big)_{\lambda_2} \Big)
        \\
        &\qquad + \sum_{\lambda_2 \in \Z} \lambda_2 \dim \Big( \bigoplus_{\lambda_1 \in \Z} \big( \mathcal{F}^{\lambda_1}S_k / \mathcal{F}^{\lambda_1 + 1}S_k \big)_{\lambda_2} \Big),
        \\
        &= \sum_{\lambda_1 \in \Z} \lambda_1 \dim \big( \mathcal{F}^{\lambda_1} S_k / \mathcal{F}^{\lambda_2 + 1} S_k \big)
        + \sum_{\lambda_2 \in \Z} \lambda_2 \dim \big( (S_k)_{\lambda_2} \big),
        \\
        &= w_\mathcal{F}(k)
        + w_\beta(k).
    \end{align*}
    using that there is an isomorphism $\bigoplus_{\lambda_1 \in \Z} ( \mathcal{F}^{\lambda_1}S_k / \mathcal{F}^{\lambda_1 + 1}S_k)_{\lambda_2} \cong (S_k)_{\lambda_2}$.
    Following a similar computation for the square-weight, we obtain
    \begin{equation*}
    \begin{aligned}
        w^2_{\mathcal{F}_\beta}(k)
        &= \sum_{\lambda_1 \in \Z} \lambda_1^2 \dim ( \mathcal{F}^{\lambda_1} S_k / \mathcal{F}^{\lambda_2 + 1} S_k)
        \\
        &\quad+ 2\sum_{\lambda_1,\lambda_2 \in \Z} \lambda_1 \lambda_2 \dim \big( ( \mathcal{F}^{\lambda_1}S_k / \mathcal{F}^{\lambda_1 + 1}S_k)_{\lambda_2} \big)
        \\
        &\quad+ \sum_{\lambda_2 \in \Z} \lambda_2^2 \dim \big( (S_k)_{\lambda_2} \big),
        \\
        &= w_\mathcal{F}^2(k) + 2w^2_{\mathcal{F},\beta}(k) + w^2_\beta(k),
    \end{aligned}
    \end{equation*}
    as desired.
\end{proof}
A modification of \cite[Lemma 3.3]{nystrom2012test} and \cite[Theorem A]{boucksom2011okounkov} gives rise to the following relation and asymptotic formula between the weight pieces, lattice functions and concave transforms:
\begin{lemma}\label{lem : T expansion for invariant filtrations by C actions}
    For any function $T : \Z^2 \to \R$, we have
    \begin{equation*}
        \sum_{\lambda, \mu \in \Z} T(\lambda,\mu) \dim \big((\mathcal{F}^\lambda S_k / \mathcal{F}^{\lambda+1} S_k)_\mu \big)
        = \sum_{a \in P_k} T(g_{\mathcal{F},k}(a), g_{\beta,k}(a)).
    \end{equation*}
    Moreover, for any continuous function $T : \R^2 \to \R$ such that $T \circ (G_\mathcal{F},G_\beta)$ is integrable on $P$ with respect to the Lebesgue measure, we have
    \begin{equation*}
        \lim_{k \to \infty} \frac{\sum_{a \in P_k} T \big(\frac{g_{\mathcal{F},k}(a)}{k}, \frac{g_{\beta,k}(a)}{k}\big)}{k^n}
        = \int_P T \circ (G_\mathcal{F}, G_\beta).
    \end{equation*}
\end{lemma}
\begin{proof}
    From \cite[Equation 1]{nystrom2012test} and the definition of the lattice functions, we have $\dim ((\mathcal{F}^\lambda R_k)_\mu) = \sum_{a \in P_k } \delta_a( g_{\mathcal{F},k} \geq \lambda ) \delta_a ( g_{\beta,k} = \mu)$ where $\delta_a$ is the Dirac measure on $P_k$ concentrated at a point $a$.
    Thus,
    \begin{align*}
        \sum_{\lambda, \mu \in \Z} &T(\lambda,\mu) \dim \big((\mathcal{F}^\lambda S_k / \mathcal{F}^{\lambda+1} S_k)_\mu \big),
        \\
        &= \sum_{\lambda, \mu \in \Z}  T(\lambda,\mu) \sum_{a \in P_k } \big( \delta_a( g_{\mathcal{F},k} \geq \lambda )
        - \delta_a( g_{\mathcal{F},k} \geq \lambda+1 ) \big) \delta_a ( g_{\beta,k} = \mu),
        \\
        &= \sum_{a \in P_k} \sum_{\lambda, \mu \in \Z} T(\lambda,\mu) \delta_a( g_{\mathcal{F},k} = \lambda ) \delta_a ( g_{\beta,k} = \mu),
        \\
        &= \sum_{a \in P_k} T(g_{\mathcal{F},k}(a), g_{\beta,k}(a)).
    \end{align*}
    
    The subset $P(t,\mathcal{F})$ in $P$ is the Okounkov subbody of the graded subalgebra $S(t,\mathcal{F}) = \bigoplus_{k \geq 0} S_k(t,\mathcal{F})$ for $S_k(t,\mathcal{F}) = \mathcal{F}^{\lceil t k \rceil} S_k$. In the case where $T$ is continuously differentiable, writing $DT(t,s) : \R^2 \to \R$ for the total derivative of $T$ at a point $(t,s)$, we have
    \begin{align*}
        \int_P T \circ (G_\mathcal{F}, G_\beta)
        &= \int_{\R^2} DT(t,s)( \mathrm{vol}(G_\mathcal{F} \geq t), \mathrm{vol}(G_\beta \geq s)),
        \\
        &= \int_{\R^2} DT(t,s)( \mathrm{vol}(P(t,\mathcal{F})) , \mathrm{vol}(P(t,\beta))),
        \\
        &= \lim_{k \to \infty} \int_{\R^2} DT(t,s) \Big( \frac{\dim S_k(t,\mathcal{F}) }{k^n} , \frac{ \dim S_k(t,\beta) }{k^n} \Big),
        \\
        &= \lim_{k \to \infty} \sum_{\lambda,\mu \in \Z} T \Big( 
            \frac{\lambda}{k}, \frac{\mu}{k}
        \Big)
        \frac{\dim ( (\mathcal{F}^\lambda S_k / \mathcal{F}^{\lambda+1} S_k)_\mu )}{k^n},
        \\
        &= \lim_{k \to \infty} \frac{\sum_{a \in P_k} T \big(\frac{g_{\mathcal{F},k}(a)}{k}, \frac{g_{\beta,k}(a)}{k}\big)}{k^n},
    \end{align*}
    where we used \cite[Equation 1.8]{boucksom2011okounkov} and \cite[Theorem 2.13]{lazarsfeld2009convex}.
    Note that we were allowed to change the limit with the integral as the volume of $P(t,\mathcal{F})$ is bounded uniformly by the volume of $P$.
    The result for general continuous functions follows by a density argument.
\end{proof}
In particular, the mixed square-weight function between $\mathcal{F}$ and $\beta$ satisfies
\begin{equation}\label{eq : weight squared function as sum of gFk and gbetak functions}
\begin{aligned}
    w_{\mathcal{F},\beta}^2(k)
    &= \Big( \int_P G_\mathcal{F} G_\beta \Big) k^{n+2} + o(k^{n+2}),
\end{aligned}
\end{equation}
as $k \to +\infty$, and the $L^2$-inner product of $\mathcal{F}$ and $\beta$ may be expressed as
\begin{equation*}
    \langle \mathcal{F}, \beta \rangle
    = \frac{1}{\mathrm{vol}(P)} \int_P (G_\mathcal{F} - \overline{G}_\mathcal{F}) (G_\beta - \overline{G}_\beta).
\end{equation*}

\subsubsection{Torus-invariant filtrations}

We conclude the section with the notion of torus-invariant filtrations and the associated reduced $L^2$-norm.
Let $T$ be a complex torus of rank $d$ in the automorphism group $\mathrm{Aut}(Y,H)$, and write
\begin{equation*}
    N_\Z(T) = \mathrm{Hom}(\C^*,T).
\end{equation*}
The group $N_\Z(T)$ is free abelian of rank $d$ whose elements are $\C^*$-actions on $(Y,H)$. We write the group operation in $N_\Z(T)$ additively.
\begin{definition}
    A filtration $\mathcal{F}$ of $(Y,H)$ is called \textit{$T$-invariant}, if all of its filtration pieces $\mathcal{F}^\lambda S_k$ are invariant under the natural linear $T$-action on $S_k$.
\end{definition}
A filtration $\mathcal{F}$ is $T$-invariant if and only if it is $\beta$-invariant for all $\beta$ in $N_\Z(T)$.
Furthermore, the inner product $\langle \mathcal{F}, \beta \rangle$ extends bilinearly to the real vector space $N_\R(T) = N_\Z(T) \otimes \R$.
This allows us to make the following definition:
\begin{definition}
    The \textit{reduced norm} of a $T$-invariant filtration $\mathcal{F}$ of $(Y,H)$ is given by the infimum
    \begin{equation*}
        ||\mathcal{F}||_T^2
        = \inf_{\xi \in N_\R(T)} \big( ||\mathcal{F}||^2
        + 2\langle \mathcal{F},\xi \rangle
        + ||\xi||^2 \big),
    \end{equation*}
    where $||\xi||^2 = \langle \xi,\xi \rangle$ is the natural extension of the norm to $N_\R(T)$.
\end{definition}
In light of Lemma \ref{lem : weight function of twisted filtration}, the expression in the infimum may be thought of as the norm of the twist $\mathcal{F}_\xi$ by a vector $\xi$ as an $\R$-filtration (see \cite[Definition 2.28]{li2022g}).
The infimum is actually a minimum:
\begin{lemma}\label{lem : red norm is minimum}
    We have
    \begin{equation*}
        ||\mathcal{F}||_T^2
        = ||\mathcal{F}||^2
        - 2\langle \mathcal{F},\xi(\mathcal{F}) \rangle
        + ||\xi(\mathcal{F})||^2,
    \end{equation*}
    for $\xi(\mathcal{F}) = \sum_{i = 1}^d \frac{\langle \mathcal{F},\beta_i \rangle}{\langle \beta_i, \beta_i \rangle} \beta_i$ where $(\beta_i)$ is any $\Z$-basis of $N_\Z(T)$.
\end{lemma}
\begin{proof}
    For any vector $\zeta$ in $N_\R(T)$, we have $\langle \xi(\mathcal{F}),\zeta \rangle = \langle \mathcal{F},\zeta \rangle$, implying
    \begin{equation*}
        ||\xi(\mathcal{F}) + \zeta||^2
        -2\langle \mathcal{F},\xi(\mathcal{F}) + \zeta \rangle
        =
        ||\xi(\mathcal{F})||^2
        + ||\zeta||^2
        -2 \langle \mathcal{F},\xi(\mathcal{F}) \rangle.
    \end{equation*}
    This is bounded from below by $||\xi(\mathcal{F})||^2 -2 \langle \mathcal{F},\xi(\mathcal{F}) \rangle$, as desired.
\end{proof}

\subsubsection{Good filtrations}

An issue with non-finitely-generated filtrations is that, in general, we have no control over subleading coefficients of the weight function.
This makes it difficult to define the polynomial Donaldson--Futaki invariant of general filtrations.
To fix this, we introduce a subclass of filtrations for which we have control over the weight function as a polynomial.
\begin{definition}[{\cite[Definition 2.14]{hattori2024minimizing}}]
    A filtration $\mathcal{F}$ of $(Y,H)$ is called \textit{good} if there is a polynomial $b(t)$ over $\Q$ such that $w_\mathcal{F}(k) = b(k)$ for sufficiently large integers $k$.
    The \textit{Donaldson--Futaki invariant} of a good filtration $\mathcal{F}$ is defined to be
    \begin{equation*}
        \mathrm{DF}(\mathcal{F}) = \frac{b_0 a_1 - b_1 a_0}{a_0^2},
    \end{equation*}
    where we have written $h^0(Y,kH) = a_0 k^n + a_1 k^{n-1} + O(k^{n-2})$ as $k \to +\infty$.
\end{definition}
Finitely generated filtrations are good and the definition the Donaldson--Futaki invariant coincides with the usual definitions for test configurations.
Using the weight formula in Lemma \ref{lem : weight function of twisted filtration}, we obtain the following:
\begin{lemma}\label{lem : DF of twist}
    If $\mathcal{F}$ is good and invariant with respect to a $\C^*$-action $\beta$ on $(Y,H)$, then
    \begin{equation*}
        \mathrm{DF}(\mathcal{F}_\beta)
        = \mathrm{DF}(\mathcal{F})
        + \mathrm{Fut}(\beta),
    \end{equation*}
    where $\mathrm{Fut}(\beta) = \mathrm{DF}(\beta)$ is the Futaki invariant of $\beta$ \cite{futaki1983obstruction,donaldson2002scalar}.
\end{lemma}

\section{Relative stability of extremal polarized manifolds}

Let $(Y,H)$ be a smooth polarized variety and let $T $ in $\mathrm{Aut}(Y,H)$ be a maximal complex torus of rank $d$.
The goal of this section is to prove that if $(Y,H)$ is extremal, then it is relatively K-polystable with respect to good filtrations.
This is second step outlined in Section \ref{subsec : structure of the paper}, and used in the proof of \ref{thm : main theorem}.
We start by introducing the relevant numerical invariant.
\begin{definition}
    Let $\mathcal{F}$ be a $T$-invariant good filtration of $(Y,H)$. The \textit{relative Donaldson--Futaki invariant} of $\mathcal{F}$ of $(Y,H)$ is defined as
    \begin{equation*}
        \mathrm{DF}_T(\mathcal{F})
        = \mathrm{DF}(\mathcal{F})
        - \sum_{i = 1}^{d} \frac{\langle \mathcal{F},\beta_i \rangle}{\langle \beta_i,\beta_i \rangle} \mathrm{Fut}(\beta_i),
    \end{equation*}
    where $(\beta_i)$ is any choice of $\Z$-basis of $N_\Z(T) = \mathrm{Hom}(\C^*,T)$.
\end{definition}
This extends the definition of the relative Donaldson--Futaki invariant for test configurations \cite[Equation 1.1]{szekelyhidi2007extremal}.

We main goal of this section is the following:
\begin{theorem}\label{thm : realtive good filtration K stab of extr}
    If $(Y,H)$ is extremal, then for any $T$-invariant good filtration $\mathcal{F}$ of $(Y,H)$, we have
    \begin{equation*}
        \mathrm{DF}_T(\mathcal{F}) \geq 0,
    \end{equation*}
    where the inequality is strict if and only if $||\mathcal{F}||_T > 0$.
\end{theorem}
This is a simultaneous generalization of \cite[Theorem 2.18]{hattori2024minimizing}, \cite[Theorem 10]{szekelyhidi2015filtrations} and \cite[Theorem 4]{stoppa2011relative} allowing automorphisms and filtrations.
In the case where $(Y,H)$ admits a constant scalar curvature K\"ahler (cscK) metric, the Futaki invariant in Lemma \ref{lem : DF of twist} vanishes (due to \cite{futaki1983obstruction}), giving rise to the following consequence:
\begin{corollary}\label{cor : good filtration k polystab of csck}
    If $(Y,H)$ admits a cscK metric, then for any $T$-invariant good filtration $\mathcal{F}$ of $(Y,H)$, we have
    \begin{equation*}
        \mathrm{DF}(\mathcal{F}) \geq 0,
    \end{equation*}
     where the inequality is strict if and only if $||\mathcal{F}||_T > 0$.
\end{corollary}
The general strategy for proving Theorem \ref{thm : realtive good filtration K stab of extr} involves first showing that the relative Donaldson--Futaki invariant is bounded from below by the relative asymptotic Chow weight (Proposition \ref{prop: DFT geq ChowinfT}) and proving that extremal manifolds are relatively asymptotically filtration Chow stable (Theorem \ref{thm : extremal mfds are rel asymp filtr chow stab}).
For the latter, we follow a method originating with Stoppa \cite{stoppa2007unstable} by first establishing a blowup formula for the relative asymptotic Chow weight, then finding a destabilizing point to blowup and then applying results by Arezzo--Pacard--Singer \cite[Theorem 2.1]{arezzo2011extremal} and Seyyedali \cite[Corollary 1.2]{seyyedali2017relative} on relative asymptotic Chow semistability of blowups of extremal manifolds (see also \cite{Mab18,ST21,Has21}) to obtain a contradiction.

\subsection{Sz\'ekelyhidi's formula}

We start with recalling Szk\'ekelyhidi's blowup formula for the Chow weight of general filtrations in \cite[Theorem 10]{szekelyhidi2015filtrations}.
Given a filtration $\mathcal{F}$ of $(Y,H)$, the \textit{$r$\textsuperscript{th}-Chow weight} of $\mathcal{F}$ is defined as
\begin{equation*}
    \mathrm{Chow}_r(\mathcal{F})
    = \frac{rb_0}{a_0} - \frac{w_\mathcal{F}(r)}{h^0(Y,rH)},
\end{equation*}
where $a_0$ and $b_0$ are as in Equation \eqref{eq : expansion for a_0 b_0 c_0}.

Choose a point $p$ in $Y$ and write $(\widehat{Y},\widehat{H}_m) = (\mathrm{Bl}_p Y,mH - E)$ for the blowup with exceptional divisor $E$, which is a polarized variety for all $m \geq m_0$ for a sufficiently large integer $m_0$.
Given a filtration $\mathcal{F}$ of $(Y,H)$, we define a new filtration $\widehat{\mathcal{F}}$ of $(\widehat{Y}, \widehat{H}_m)$ by
\begin{equation}\label{eq : definition of Fhat}
    \widehat{\mathcal{F}}^\lambda H^0 \big(\widehat{Y},k\widehat{H}_m \big) = \{ \sigma \;|\; i(\sigma) \in \mathcal{F}^\lambda H^0(Y,kmH) \},
\end{equation}
where $i : H^0(\widehat{Y},k\widehat{H}_m) \hookrightarrow H^0(Y,kmH)$ is the natural embedding of sections.
This filtration generalizes the construction of a test configuration on blowups of Stoppa \cite[Section 2]{stoppa2007unstable}.

Given a choice of Okounkov body $P$ for $(Y,H)$ centered at $p$, constructed out of data $(U,z,\sigma)$, one obtains an Okounkov body
\begin{equation*}
    \widehat{P} = \overline{mP \setminus \Delta_1},
\end{equation*}
for $(\widehat{Y},\widehat{H}_m)$ constructed out of lifted data $(\widehat{U},\widehat{z},\widehat{\sigma^m})$, where $\Delta_1 \subset mP$ (see \cite[Lemma 6]{szekelyhidi2015filtrations}).
A direct computation shows that the concave transform of $\widehat{\mathcal{F}}$ on $\widehat{P}$ satisfies $G_{\widehat{\mathcal{F}}}(x) = m G_\mathcal{F} (\frac{x}{m})$ for all $x$ in $\widehat{P}$.

From now on, we will assume $n \geq 2$. We will refer to the case $n \leq 1$ in Remark \ref{rmk : case n = 1} when discussing the stability results.
Sz\'ekelyhidi's formula states that the Chow weight of the filtration $\widehat{\mathcal{F}}$ on the blowup is asymptotically bounded by the Chow weight of the tautological approximation $(\mathcal{F}_{(r)})$ of $\mathcal{F}$ (see Example \ref{exp : approx filtrations}) in the following way:
\begin{theorem}[{\cite[Lemma 12]{szekelyhidi2015filtrations}}]\label{thm : szekelyhidis asymptotic}
    Let $P$ be an Okounkov body of $(Y,H)$ centered at $p$ and assume that $G_\mathcal{F} \leq \Lambda$ on $\Delta_{\frac{1}{m}}$ for a real number $\Lambda$, where $\Delta_l$ denotes the solid closed $n$-simplex of length $l$ in $\R^n$.
    Then, 
    \begin{equation*}
        w_{\widehat{\mathcal{F}}}(r)
        \leq w_{\mathcal{F}_{(mr)}}(1)
        - mr^n \int_{\Delta_{\frac{1}{m}-\frac{n}{r}}} G_\mathcal{F}
        - \frac{\Lambda(3n-1)}{2(n-1)!}m^{1-n}r^{n-1}
        + O(m^0r^{n-2})
    \end{equation*}
    as $r \to +\infty$ and subsequently $m \to +\infty$. In particular, we have
    \begin{equation*}
        \mathrm{Chow}_r (\widehat{\mathcal{F}})
        \leq \mathrm{Chow}_{1}(\mathcal{F}_{(rm)})
        + \Big( \nu(p,\mathcal{F}) - \frac{b_0}{a_0} \Big) \frac{m^{1-n}}{2(n-2)!a_0}
        + o(m^{-n} r^0)
    \end{equation*}
   as $r \to +\infty$ and $m \to +\infty$, where
   \begin{equation*}
       \nu(p,\mathcal{F})
       = (n-2)!a_0 \bigg(
       \frac{\Lambda(3n-1)}{2(n-1)!}
       - \frac{r}{m^{1-n}}\int_{\Delta_{\frac{1}{m}} \setminus \Delta_{\frac{1}{m}-\frac{n}{r}}} G_\mathcal{F}
       \bigg).
   \end{equation*}
\end{theorem}
In the proof of Theorem \ref{thm : extremal mfds are rel asymp filtr chow stab}, we will use specific choices of $p$ and $\Lambda$, coming from Corollary \ref{cor : relative appendix corollary}.

\subsection{A relative formula.}

We next establish the a relative blowup formula, extending Theorem \ref{thm : szekelyhidis asymptotic}, which will be needed in the proof of Theorem \ref{thm : extremal mfds are rel asymp filtr chow stab}.
Given a $\C^*$-action $\beta$ on $(Y,H)$ and a $\beta$-invariant filtration $\mathcal{F}$ be $(Y,H)$, we define
\begin{equation}\label{eq : rth inner product}
    \langle \mathcal{F}, \beta \rangle_r
    = \frac{w^2_{\mathcal{F},\beta}(r)h^0(Y,rH) - w_\mathcal{F}(r) w_\beta(r)}{r^2 h^0(Y,rH)^2}.
\end{equation}
By definition of the $L^2$-inner product, we have
\begin{equation*}
     \langle \mathcal{F}, \beta \rangle
     = \lim_{r \to +\infty} \langle \mathcal{F}, \beta \rangle_r.
\end{equation*}

Assume now that $p$ is fixed by $\beta$. In this case, the $\C^*$-action $\beta$ naturally lifts to a $\C^*$-action $\widehat{\beta}$ on the blowup $(\widehat{Y},\widehat{H}_m)$ leaving $\widehat{\mathcal{F}}$ invariant.
\begin{proposition}\label{prop : asymptotic for rth inner products}
    As $r \to +\infty$ and $m \to +\infty$, we have
    \begin{equation*}
        \langle \widehat{\mathcal{F}}, \widehat{\beta} \rangle_r
        = \langle \mathcal{F}, \beta \rangle_{mr}m^2
        + o(m^{2-n}r^0).
    \end{equation*}
\end{proposition}
\begin{proof}
    By the asymptotic Riemann-Roch formula we have
    \begin{equation*}
        h^0(\widehat{Y}, r\widehat{H}_m)
        = h^0(Y,mrH) + O(m^0 r^n)
    \end{equation*}
    as $r \to +\infty$ and $m \to +\infty$.
    
    Using $G_{\widehat{\mathcal{F}}}(x) = m G_\mathcal{F} (\frac{x}{m})$ on $\widehat{P} = \overline{mP \setminus \Delta_1}$ and Equation \eqref{eq : weight function as sum of gk functions}, we have
    \begin{equation*}
    \begin{aligned}
        w_{\widehat{\mathcal{F}}}(r)
        &= \bigg(
            \int_{mP} G_\mathcal{F}
        \bigg) m^{n+1} r^{n+1}
        -  \bigg(
            \int_{\Delta_{\frac{1}{m}}} G_\mathcal{F}
        \bigg) m^{n+1} r^{n+1}
        + o(r^{n+1}),
        \\
        &= w_\mathcal{F}(mr)
        + O(m r^{n+1}).
    \end{aligned}
    \end{equation*}

    Similarly, using Equation \eqref{eq : weight squared function as sum of gFk and gbetak functions}, we have
    \begin{equation*}
    \begin{aligned}
        w^2_{ \widehat{\mathcal{F}},\widehat{\beta}}(r)
        &= \bigg( \int_P G_\mathcal{F} G_\beta \bigg) m^{n+2}r^{n+2} - \bigg( \int_{\Delta_{\frac{1}{m}}} G_\mathcal{F} G_\beta \bigg) m^{n+2}r^{n+2}
        + o(r^{n+2}),
        \\
        &= w^2_{\mathcal{F},\beta}(mr)
        + O(m^2r^{n+2}).
    \end{aligned}
    \end{equation*}
    Putting these expansions in Equation \eqref{eq : rth inner product} and taking the limits in the right order yields the result.
\end{proof}
\begin{remark}
    Letting $r \to +\infty$, we obtain
    \begin{equation*}
        \langle \widehat{\mathcal{F}}, \widehat{\beta} \rangle
        = \langle \mathcal{F}, \beta \rangle m^2 + O(m^{2-n})
    \end{equation*}
    as $m \to +\infty$. This matches the formula proven by Stoppa--Szk\'ekelyhidi \cite[Page 8]{stoppa2011relative} for test configurations.
\end{remark}

Given an $T$-invariant filtration $\mathcal{F}$ of $(Y,H)$ for a torus $T$ in $\mathrm{Aut}(Y,H)$, we define relative $r$\textsuperscript{th}-Chow weight
\begin{equation}\label{eq : def rel weight and rel chow weight}
    \mathrm{Chow}_{r,T}(\mathcal{F})
    = \mathrm{Chow}_r(\mathcal{F})
    - \sum_{i = 1}^d \frac{\langle \mathcal{F}, \beta_i \rangle_r}{\langle \beta_i, \beta_i \rangle_r} \mathrm{Chow}_r(\beta_i),
\end{equation}
where $(\beta_i)$ is any choice of $\Z$-basis of $N_\Z(T)$. In the case of test configurations, Equation \eqref{eq : def rel weight and rel chow weight} captures the weight relevant for relative Chow stability.
\begin{lemma}\label{lem : rel chow of approx leq rel chow of F}
    We have,
    \begin{equation*}
        \mathrm{Chow}_{1,T}(\mathcal{F}_{(r)})
        \leq \mathrm{Chow}_{r,T}(\mathcal{F}).
    \end{equation*}
    In particular, if $\mathrm{Chow}_{r,T}(\mathcal{F})$ is negative for infinitely many $r$, then $(Y,H)$ is asymptotically relatively Chow unstable with respect to $T$.
\end{lemma}
\begin{proof}
    From \cite[Lemma 9]{szekelyhidi2015filtrations}, we know the non-relative case
    \begin{equation*}
        \mathrm{Chow}_{1}(\mathcal{F}_{(r)})
        \leq \mathrm{Chow}_{r}(\mathcal{F}).
    \end{equation*}
    This is an equality whenever $\mathcal{F}$ is finitely generated because, in that case, the approximation $\mathcal{F}_{(r)}$ coincides with $\mathcal{F}$, viewed as a filtration of $(X,rL)$.
    Moreover, as $\mathcal{F}^\lambda_{(r)} S_r = \mathcal{F}^\lambda S_r$ for any filtration $\mathcal{F}$, the inner product satisfies $\langle \mathcal{F}_{(r)}, \beta_i \rangle_1 = \langle \mathcal{F},\beta_i \rangle_r$. Putting everything together proves the result.
\end{proof}

Assume now that $p$ is fixed by $T$. We obtain a torus $\widehat{T}$ in $\mathrm{Aut}(\widehat{Y},\widehat{H}_m)$ which leaves $\widehat{\mathcal{F}}$ invariant.
Combining Theorem \ref{thm : szekelyhidis asymptotic} and Proposition \ref{prop : asymptotic for rth inner products}, we obtain the following relative formula:
\begin{theorem}\label{thm : the relative filtration asymptotic}
    Assume that $G_\mathcal{F} \leq \Lambda$ on $\Delta_{\frac{1}{m}}$ as in Theorem \ref{thm : szekelyhidis asymptotic}. Then,
    \begin{equation*}
        \mathrm{Chow}_{r,T} (\widehat{\mathcal{F}})
        \leq \mathrm{Chow}_{1,T}(\mathcal{F}_{(mr)})
        + \Big( \nu_T(p,\mathcal{F}) - \frac{b_0}{a_0} \Big) \frac{m^{1-n}}{2(n-2)!a_0}
        + o(m^{-n} r^0)
    \end{equation*}
    as $r \to +\infty$ and $m \to +\infty$, where
    \begin{equation}\label{eq : realtive lambda weight}
        \nu_T(p,\mathcal{F})
        = \nu(p,\mathcal{F})
        - \sum_{i = 1}^d \frac{\langle \mathcal{F}, \beta_i \rangle}{\langle \beta_i, \beta_i \rangle} \Big( \nu(p,\beta_i) - \frac{b_{i,0}}{a_0} \Big),
    \end{equation}
    given any $\Z$-basis $(\beta_i)$ of $N_\Z(T)$ and $w_{\beta_i}(k) = b_{i,0}k^{n+1} + O(k^n)$.
\end{theorem}

\subsection{Finding destabilizing points.}

In order to use the formula in Proposition \ref{thm : the relative filtration asymptotic} for an instability argument on the blowup, we need to find an invariant point $p$ for which the relative weight in Equation \eqref{eq : realtive lambda weight} is small enough for large $m$.
For this, we establish a $T$-invariant version of results of Boucksom \cite[Section 8]{szekelyhidi2015filtrations}.

Let $S'$ be a graded subalgebra of the ring of sections $S = \bigoplus_{k \geq 0} S_k$ for $S_k = H^0(Y,kH)$, not equal to $\C$. Recall the volume of $S'$ is the limit
\begin{equation*}
    \mathrm{vol}(S') = \lim_{\begin{smallmatrix}
        k \to \infty \\ S'_k \neq 0
    \end{smallmatrix}} \frac{\dim S'_k}{k^\kappa / \kappa!},
\end{equation*}
where $\kappa = \mathrm{trdeg}(S'/\C) - 1$ is the Iitaka dimension of $S'$.
\begin{definition}
    A graded subalgebra $S'$ of $S$ is said to \textit{contain an ample series} if there exists a decomposition $L = A + E$ of $\Q$-divisors with $A$ ample and $E$ effective such that $H^0(X,kA) \subset S_k$ for all sufficiently divisible $k$.
\end{definition}
The main result of this subsection is the following:
\begin{theorem}\label{thm : ample series theorem}
    If $S'$ is $T$-invariant, contains an ample series and satisfies $\mathrm{vol}(S') < \mathrm{vol}(L)$, then there is a $T$-invariant point $p$ in $X$ and an integer $m_1 \geq 1$ such that
    \begin{equation*}
        S_{mk}' \subset H^0(X,mkL \otimes \mathcal{I}^{\otimes k}_p),
    \end{equation*}
    for all $m \geq m_1$ and $k \geq 0$ where $\mathcal{I}_p$ is the ideal sheaf of $p$ in $X$.
\end{theorem}
This is a relative version of \cite[Theorem 20]{szekelyhidi2015filtrations} and its proof is based on the following relative versions of \cite[Lemma 22]{szekelyhidi2015filtrations} and \cite[Lemma 21]{szekelyhidi2015filtrations}:
\begin{lemma}\label{lem : Bouksom Lemma 1}
    Let $S'$ and $S''$ be two $T$-invariant graded subalgebras of $S$ such that $v(S') \geq v(S'')$ for all $T$-invariant divisorial valuations $v$ of $Y$. Then $\mathrm{vol}(S') \leq \mathrm{vol}(S'')$.
\end{lemma}
\begin{lemma}\label{lem : Bouksom Lemma 2}
    If there is a $T$-invariant divisorial valuation $v$ on $X$ such that $v(S') > 0$, then there is a $T$-invariant point $p$ in $X$ such that $\ord_p(S') > 0$.
\end{lemma}
A divisorial valuation $v$ of $Y$ is a valuation $v : \C(Y)^* \to \R$ of the form $v = c \cdot \ord_E$ for a real number $c > 0$ and a prime divisor $E$ over $Y$, and we define
\begin{equation*}
    v(S') = \lim_{k \to +\infty} \frac{\min \{ v(f) \;|\; f \in \mathfrak{bs}(S_k') \setminus \{0\} \}}{k},
\end{equation*}
where $\mathfrak{bs}(S_k')$ is the base-ideal of the linear series $S_k'$ in $S$.
The proofs of Lemmas \ref{lem : Bouksom Lemma 1} and \ref{lem : Bouksom Lemma 2} follow \textit{verbatim} the proofs of \cite[Lemma 22]{szekelyhidi2015filtrations} and \cite[Lemma 21]{szekelyhidi2015filtrations} noting that the divisorial valuations and their centers constructed in the proofs are invariant. We therefore omit the details.
\begin{proof}[Proof of Theorem \ref{thm : ample series theorem}]
    By Lemma \ref{lem : Bouksom Lemma 1}, there exists a $T$-invariant divisorial valuation $v$ on $X$ such that $v(S') >  v(S)$.
    Using that $v(S) = 0$, Lemma \ref{lem : Bouksom Lemma 2} implies that there exists a $T$-invariant point $p$ in $X$ such that $\mathrm{ord}_p(S') > 0$.
    But this simply means $S_{mk}' \subset H^0(Y,mkH \otimes \mathcal{I}^{\otimes k}_p)$ for sufficiently large $m$ and all $k$, as desired.
\end{proof}
\begin{remark}
    For test configurations, the existence and construction of a destabilizing point is proven differently in Stoppa--Sz\'ekelyhidi \cite{stoppa2011relative}, using arguments involving embeddings into projective space.
\end{remark}
Let $\mathcal{F}$ be a filtration $(Y,H)$ and write
\begin{equation*}
    e_\mathrm{max}(\mathcal{F})
    := \lim_{k \to \infty} \frac{\sup\{ \lambda \in \Z \;|\; \mathcal{F}^\lambda R_k \neq 0 \}}{k}.
\end{equation*}
This value is finite and equals the essential supremum of the concave transform $G_\mathcal{F}$ for any choice of Okounkov body \cite[Section 1]{boucksom2011okounkov}.
The desired property that the invariant point in Theorem \ref{thm : ample series theorem} satisfies is the following: 
\begin{corollary}\label{cor : relative appendix corollary}
    Let $\Lambda < e_\mathrm{max}(\mathcal{F})$ be a real number. There exists a $T$-invariant point $p$ in $X$ such that for any Okounkov body $P$ centered at $p$, there is an integer $m_1 \geq m_0$ satisfying
    \begin{equation*}
        G_\mathcal{F}|_{\Delta_{\frac{1}{m_1}}} \leq \Lambda.
    \end{equation*}
\end{corollary}
\begin{proof}
    From \cite[Lemma 1.6]{boucksom2011okounkov} we know that if $\Lambda < e_\mathrm{max}(\mathcal{F})$, then the graded subalgebra $S(\Lambda,\mathcal{F}) = \bigoplus_{k \geq 0} \mathcal{F}^{\lceil \Lambda k \rceil} S_k$ contains an ample series.
    Moreover, if $\mathcal{F}$ is $T$-invariant, then is $S(\Lambda,\mathcal{F})$ is $T$-invariant by definition.
    Applying Proposition \ref{thm : ample series theorem} to $S' = S(\Lambda,\mathcal{F})$, we obtain the desired upper bound for the concave transform.
\end{proof}

\subsection{Stability results}

In this section, we prove the relevant stability result, Theorem \ref{thm : realtive good filtration K stab of extr}.
Given a $T$-invariant filtration $\mathcal{F}$, we define the \textit{relative asymptotic Chow weight}
\begin{equation*}
\begin{aligned}
    \mathrm{Chow}_{\infty,T}(\mathcal{F})
    = \liminf_{r \to +\infty} \mathrm{Chow}_{1,T}(\mathcal{F}_{(r)}),
\end{aligned}
\end{equation*}
where $\mathcal{F}_{(r)}$ is the filtration defined in Example \ref{exp : approx filtrations}.
\begin{proposition}\label{prop: DFT geq ChowinfT}
    For any $T$-equivariant good filtration $\mathcal{F}$ of $(X,L)$, we have
    \begin{equation*}
        \mathrm{DF}_T(\mathcal{F}) \geq \mathrm{Chow}_{\infty,T}(\mathcal{F}).
    \end{equation*}
\end{proposition}
This is a relative version of \cite[Theorem 2.18]{hattori2024minimizing}.
\begin{proof}    
    From the definition of the relative Donaldson--Futaki invariant, we have $\mathrm{DF}_T(\mathcal{F}) = \lim_{r \to \infty} \mathrm{Chow}_{r,T}(\mathcal{F})$. Combining this with Lemma \ref{lem : rel chow of approx leq rel chow of F} and the definition of the relative asymptotic Chow weight yields the desired inequality.
\end{proof}
We come to the main result of this subsection. For this, assume that the torus $T$ in $\mathrm{Aut}(Y,H)$ is maximal.
\begin{theorem}\label{thm : extremal mfds are rel asymp filtr chow stab}
    If $(X,L)$ is extremal, then for any $T$-invariant filtration $\mathcal{F}$ of $(X,L)$ with $||\mathcal{F}||_T > 0$, we have
    \begin{equation*}
        \mathrm{Chow}_{\infty,T}(\mathcal{F})
        > 0.
    \end{equation*}
\end{theorem}
This is a relative version \cite[Theorem 11]{szekelyhidi2015filtrations} allowing extremal metrics and automorphisms.
Together with Proposition \ref{prop: DFT geq ChowinfT}, it implies the desired stability result, Theorem \ref{thm : realtive good filtration K stab of extr}.
\begin{proof}
    First, note that it suffices to prove
    \begin{equation*}
        \mathrm{Chow}_\infty(\mathcal{F})
        := \liminf_{r \to +\infty} \mathrm{Chow}_1(\mathcal{F}_{(r)})
        > 0,
    \end{equation*}
    for $T$-invariant filtrations $\mathcal{F}$ with $\langle \mathcal{F},N_\Z(T) \rangle = 0$ and $||\mathcal{F}|| > 0$.
    This may be done by replacing $\mathcal{F}$ with the twist $\mathcal{F}_{-\xi(\mathcal{F})}$ for $\xi(\mathcal{F)}$ in $N_\R(T)$ as defined in Lemma \ref{lem : red norm is minimum}.
    There is a slight subtlety that $\mathcal{F}_{-\xi(\mathcal{F})}$ is an $\R$-filtration, in general; we ignore this and hence assume that $\xi(\mathcal{F})$ is rational, as the argument applies in the same way in general, and $\xi(\mathcal{F})$ is rational in our applications.
    
    Now, assume by contradiction that $\mathrm{Chow}_\infty(\mathcal{F}) = 0$. As $||\mathcal{F}|| > 0$, the essential infimum $M_\mathcal{F}$ of the concave transform $G_\mathcal{F}$ is strictly smaller than its average $\overline{G}_\mathcal{F}$ over $P$.
    Note that the values $M_\mathcal{F}$ and $\overline{G}_\mathcal{F}$ are both independent of choice of Okounkov body.
    Let
    \begin{equation}\label{eq : def Lambda F}
        \Lambda_\mathcal{F}
        = \frac{A_n-2}{A_n} \overline{G}_\mathcal{F} + \frac{1}{A_n}M_\mathcal{F},
    \end{equation}
    for $A_n = \frac{5n-1}{2n}$.
    This number will serve as the value in Corollary \ref{cor : relative appendix corollary} to obtain an upper bound for the concave transform near the origin.
    Indeed, since $M_\mathcal{F} < \overline{G}_\mathcal{F}$, we have $\Lambda_\mathcal{F} < e_\mathrm{max}(\mathcal{F})$ so by Corollary \ref{cor : relative appendix corollary}, there exists a $T$-invariant point $p$ in $X$ such that, after choosing an Okounkov body $P$ centered around $p$, there is an integer $m_0 \geq 1$, such that
    \begin{equation*}
        G_\mathcal{F}|_{\Delta_{\frac{1}{m_0}}} \leq \Lambda_\mathcal{F}.
    \end{equation*}

    By assumption, for any $\delta > 0$ and $s_1 \geq 1$, there exist infinitely many integers $s \geq s_1$ such that $\mathrm{Chow}_1 (\mathcal{F}_{(s)}) < \delta$.
    We will proceed to set up appropriate choices of $\delta$ and $s_1$:
    let $\delta$ be a real number satisfying
    \begin{equation*}
        0 < \delta
        < \frac{\overline{G}_\mathcal{F} - M_\mathcal{F}}{A_n B_n},
    \end{equation*}
    where $B_n = \frac{7n-1}{2n}$.
    By \cite[Lemma 7]{szekelyhidi2015filtrations}, there exists an integer $s_0 \geq 1$ such that for all $s \geq s_0$, we have $\overline{G}_\mathcal{F} - \delta \leq \overline{G}_{\mathcal{F}_{(s)}} \leq \overline{G}_\mathcal{F}$.
    This gives rise to an a priori estimate
    \begin{equation*}
        \int_{\Delta_{\frac{1}{m}}\setminus\Delta_{\frac{1}{m}-\frac{n}{r}}} G_{\mathcal{F}_{(s)}}
        \geq \frac{nm^{1-n}}{r(n-1)!}(M_\mathcal{F} - \delta),
    \end{equation*}
    for any $r \geq 1$.
    Moreover, by \cite[Lemma 7]{szekelyhidi2015filtrations}, the integral formula of the inner product, the boundedness of the weights and orthogonality of $\mathcal{F}$ with respect to $N_\Z(T)$, there exists an integer $s_1 \geq s_0$, such that for all $s \geq 0$ and $m \geq m_0$, we have
    \begin{align*}
        \bigg| \sum_{i = 1}^d \frac{\langle \mathcal{F}_{(s)}, \beta_i \rangle}{\langle \beta_i, \beta_i \rangle} \mathrm{Chow}_m(\beta_i) \bigg|
        &\leq \delta,
        \\
        \bigg| \sum_{i = 1}^d \frac{\langle \mathcal{F}_{(s)}, \beta_i \rangle}{\langle \beta_i, \beta_i \rangle} \Big( \nu(\beta_i) - \frac{b_{i,0}}{a_0} \Big) \bigg|
        &\leq \frac{a_0 \delta}{n-1},
    \end{align*}
    where $a_0$ is from Equation \ref{eq : expansion for a_0 b_0 c_0}.
    With these choices of $\delta$ and $s_1$, there exist infinitely many integers $s \geq s_0$ such that $\mathrm{Chow}_1 (\mathcal{F}_{(s)}) < \delta$.
    Let $\mathcal{F}'$ be the $\overline{G}_\mathcal{F}$-shifted filtration of $\mathcal{F}$, given as
    \begin{equation*}
        \mathcal{F}'^{\,\lambda} H^0(Y,kH) = \mathcal{F}^{\lambda + \overline{G}_{\mathcal{F}}k} H^0(Y,kH).
    \end{equation*}
    This filtration has the same Chow weights, while the leading coefficient of its weight function vanishes.
    By taking $s$ to be large enough and changing $m$, we may assume that $\frac{s}{m} =: r$ is an integer.
    
    The relative asymptotic expansion of Theorem \ref{thm : the relative filtration asymptotic} for this filtration then takes the form
    \begin{equation*}
        \mathrm{Chow}_{r,T} (\widehat{\mathcal{F}'})
        \leq  2\delta
        + \nu_{T}(p,\mathcal{F}_{(s)}') \frac{m^{1-n}}{2(n-2)!a_0}
        + o(m^{-n} r^0).
    \end{equation*}

    We now prove that the $\nu_T$-weight is negative:
    write $\gamma_n = \frac{\overline{G}_\mathcal{F} - M_\mathcal{F}}{a_n} > 0$, which satisfies both $\Lambda_\mathcal{F} = \overline{G}_\mathcal{F} - \gamma_n$ and $M_\mathcal{F} = \overline{G}_\mathcal{F} - A_n\gamma_n$. 
    Using the relation $G_{\mathcal{F}_{(s)}'} = G_{\mathcal{F}_{(s)}} - \overline{G}_{\mathcal{F}_{(s)}}$,
    the summands appearing in the weight satisfy the estimates
    \begin{align*}
        \frac{(3n-1)}{2(n-1)!} \Lambda_{\mathcal{F}'_{(s)}}
        &\leq \frac{n}{(n-1)!}  \frac{3n-1}{2n} (\delta - \gamma_n),
        \\
        \frac{r}{m^{n-1}}\int_{\Delta_{\frac{1}{m}}\setminus\Delta_{\frac{1}{m}-\frac{n}{r}}} G_{\mathcal{F}'_{(s)}}
        &\geq -\frac{n}{(n-1)!}(\delta + A_n\gamma_n).
    \end{align*}
    Using that $G_{\mathcal{F}'_{(s)}} \leq (\delta - \gamma_n)$ on $\Delta_{\frac{1}{m_0}}$.
    Putting this into the weight, we obtain
    \begin{equation*}
        \nu_T(p,\mathcal{F}_{(s)}')
        \leq \frac{a_0}{n-1} ( b_n\delta m^{-1} - \gamma_n),
    \end{equation*}
    noting that the right-hand side is negative as $\delta < \frac{\gamma_n}{b_n}$ by construction.
    
    Altogether, we conclude that 
    \begin{equation}\label{eq : final estimate for ChrhatFsPrime}
        \mathrm{Chow}_{r,\widehat{T}}(\widehat{\mathcal{F}'})
        \leq \Big( 2\delta + \frac{a_0}{n-1}(b_n\delta - m\gamma_n) m^{-n}
        \Big)
        + O(m^{1-n}r^{-1}).
    \end{equation}
    
    Now, choose $m \geq m_1$ to be large enough such the leading term above is negative and the blowup $(\widehat{Y}, \widehat{sH}_m) = (\mathrm{Bl}_p Y, msH - E)$ is extremal \cite[Theorem 2.1]{arezzo2011extremal}.
    But this is impossible, as \cite[Corollary 1.2]{seyyedali2017relative} implies that the blowup is asymptotically relatively Chow semistable, but Equation \eqref{eq : final estimate for ChrhatFsPrime} is negative for infinitely many $r$, contradicting Lemma \ref{lem : rel chow of approx leq rel chow of F}.
\end{proof}
We used the fact that the torus is maximal in the final step in order to apply \cite{arezzo2011extremal} and \cite{seyyedali2017relative} (resp. \cite{Mab18,ST21,Has21}).

\begin{remark}\label{rmk : case n = 1}
    The result in the case $n = 1$ may be reduced to the higher-dimensional result $n = 2$ as in \cite[Page 18]{szekelyhidi2015filtrations}:
    any $T$-invariant filtration $\mathcal{F}$ of $(Y,H)$ induces a canonical $T \times T$-invariant filtration $\mathcal{F}'$ on the product $(Y \times Y, \mathrm{pr}_1^*H \otimes \mathrm{pr}_2^*H)$ satisfying $\mathrm{Chow}_{r,T \times T}(\mathcal{F}') = \mathrm{Chow}_{r,T}(\mathcal{F})$.
\end{remark}

\section{Main results}

For the final section, we prove the main results of the introduction, namely Theorem \ref{thm : main theorem} and Corollary \ref{cor : main corollary}.
For this, we first examine the relationship between models and filtrations and their twists by one-parameter subgroups.

\subsection{Models and filtrations}

We start by giving the construction of a good filtration from two models, due to Hattori \cite{hattori2024minimizing}.
Let $R$ be a discrete valuation ring with fraction field $K$ and residue field $\C$ and let $(X,L)$ be a polarized $K$-variety.
Let $(\mathcal{X},\mathcal{L})$ and $(\mathcal{X}',\mathcal{L}')$ be models of $(X,L)$ and assume that $L$, $\mathcal{L}$ and $\mathcal{L}'$ are ($\Z$-)line bundles.
 
We repeat the construction in Equation \eqref{eq : res of indeterm} by letting $\mathcal{X}' \xleftarrow{\mu'} \mathcal{Y} \xrightarrow{\mu} \mathcal{X}$ be a resolution of indeterminacies of the canonical birational map $\mathcal{X}' \dashrightarrow \mathcal{X}$ over $\Spec R$, and write $\mu^*\mathcal{L} = \mu'^*\mathcal{L}' + E - a\mathcal{Y}_0$ for an effective divisor $E$ in $\mathcal{Y}_0$ and an integer $a \geq 1$.
We obtain a natural embedding
\begin{equation}\label{eq : embedding i}
    i : H^0(\mathcal{Y},k\mathcal{M}) \to H^0(\mathcal{X},k\mathcal{L}),
\end{equation}
for $\mathcal{M} = \mu'^* \mathcal{L}' - a \mathcal{Y}_0$.
Choose a uniformizer $t$ in $R$.
\begin{definition}\label{def : definition of Hattoris good filtration}
    The \textit{good filtration associated to the models $(\mathcal{X},\mathcal{L})$ and $(\mathcal{X}',\mathcal{L}')$} is the filtration $\mathcal{F} = \mathcal{F}_{(\mathcal{X},\mathcal{L}),(\mathcal{X}',\mathcal{L}')}$ of $(\mathcal{X}_0,\mathcal{L}_0)$, given as
    \begin{equation*}
        \mathcal{F}^\lambda H^0(\mathcal{X}_0,k\mathcal{L}_0)
        = \{ s|_{\mathcal{X}_0} \;|\; 
            s \in H^0(\mathcal{X},k\mathcal{L}) \text{ and } 
            t^{-\lambda}s \in i(H^0(\mathcal{Y},k\mathcal{M}))
        \},
    \end{equation*}
    viewing $t$ as a regular function on $\mathcal{X}$ by pulling back along $\mathcal{X} \to \Spec R$.
\end{definition}
In the case where $L$, $\mathcal{L}$ and $\mathcal{L}'$ are $\Q$-line bundles, we obtain a filtration of $(\mathcal{X}_0,r\mathcal{L}_0)$ for a sufficiently divisible positive integer $r$.
The filtration was constructed in \cite[Theorem 3.5]{hattori2024minimizing} in order to prove Theorem \ref{thm : hattoris theorem}, where it was defined for families over smooth projective curves.
However, as the original construction passes through restriction to an affine neighborhood, the original construction coincides with the above definition.
Hattori's good filtration generalizes the filtration constructed by Blum--Xu \cite{blum2019uniqueness} for degenerations of Fano varieties.

The main property of Hattori's filtration is the following:
\begin{theorem}\label{thm : hattoris good filtration properties}
    The filtration $\mathcal{F} = \mathcal{F}_{(\mathcal{X},\mathcal{L}),(\mathcal{X}',\mathcal{L}')}$ is good and satisfies
    \begin{equation*}
        w_\mathcal{F}(k) = - \dim \! \big( H^0(\mathcal{X}, k\mathcal{L})/i(H^0(\mathcal{Y},k\mathcal{M})) \big),
    \end{equation*}
    for $k \gg 0$. Moreover $||\mathcal{F}|| > 0$ if and only if the models are not isomorphic.
\end{theorem}
This was proven in \cite[Theorem 3.5 and Corollary 3.9]{hattori2024minimizing} for families over smooth projective curves.
We will see that the same proof strategy applies to models over DVRs.
\begin{proof}
    We apply the split short exact sequence of coherent $\mathcal{O}_{\Spec R}$-modules
    \begin{equation*}
        0
        \to t\mathcal{V} / (\mathcal{W} \cap t\mathcal{V})
        \to \mathcal{V} / \mathcal{W}
        \to (\mathcal{V} / t\mathcal{V}) / \mathrm{im}(\mathcal{W} \to \mathcal{V}/t\mathcal{V})
        \to 0,
    \end{equation*}
    iteratively to the cases $\mathcal{V}_\lambda = t^\lambda H^0(\mathcal{X},k\mathcal{L})$ and $\mathcal{W}_\lambda = t^\lambda \mathcal{F}^{-\lambda}H^0(\mathcal{X},k\mathcal{L})$ for $\lambda \geq 0$, where $\mathcal{F}^{\lambda}H^0(\mathcal{X},k\mathcal{L}) := \{ s \;|\; t^{-\lambda} s \in i(H^0(\mathcal{Y},k\mathcal{M})) \}$.
    We have
    \begin{equation*}
    \begin{aligned}
        t\mathcal{V}_\lambda / (\mathcal{W}_\lambda \cap t \mathcal{V}_\lambda)
        &= \mathcal{V}_{\lambda + 1}/\mathcal{W}_{\lambda + 1},
        \\
        \mathcal{V}_\lambda / t\mathcal{V}_\lambda
        &= H^0(\mathcal{X}_0,k\mathcal{L}_0),
        \\
        \mathrm{im}( \mathcal{W}_\lambda \to \mathcal{V}_\lambda/ t\mathcal{V}_\lambda)
        &= t^\lambda \mathcal{F}^{-\lambda}H^0(\mathcal{X}_0,k\mathcal{L}_0),
    \end{aligned}
    \end{equation*}
    so applying the short exact sequence, we obtain
    \begin{equation}\label{eq : isomorphism for weight formula}
    \begin{aligned}
        \mathcal{V}_0 / \mathcal{W}_0
        &\simeq 
        \mathcal{V}_1 / \mathcal{W}_1
        \oplus H^0(\mathcal{X}_0,k\mathcal{L}_0) / t^0 \mathcal{F}^{-0}H^0(\mathcal{X}_0,k\mathcal{L}_0),
        \\
        &\simeq \dots \simeq
        \bigoplus_{\lambda \geq 0}  \big( H^0(\mathcal{X}_0,k\mathcal{L}_0) / ( t^\lambda \mathcal{F}^{-\lambda} H^0(\mathcal{X}_0,k\mathcal{L}_0)) \big).
    \end{aligned}
    \end{equation}
    By construction, the left-hand side is $H^0(\mathcal{X},kr\mathcal{L}) / i(H^0(\mathcal{Y},kr\mathcal{M}))$ and the dimension of the right-hand side the weight $w_\mathcal{F}(k)$.

    For the norm of the filtration, it suffices to show that if the difference $\mu^*\mathcal{L} - \mathcal{M}$ supported in $\mathcal{Y}_0$ is not $\Q$-linearly trivial, then the concave transform $G_\mathcal{F}$ is not constant.
    Similar to Section \ref{subsubsec : models and dvrs}, we may assume that $E$ is non-zero, $\mu$-exceptional and $-E$ is $\mu$-ample by taking an ample model.

    Assume that $G_\mathcal{F}$ is constant and let $f : [0,1] \to \R$ be the function
    \begin{equation*}
        f(s)
        = \mu^*\mathcal{L}^{n-1} - \mathcal{M}_s^{n+1}
    \end{equation*}
    for $\mathcal{M}_s = \mu^*\mathcal{L} - sE$,
    where the intersection number is taken in the sense of, for example, \cite[Section 2]{dervan2024arcs}, meaning that $\mu^*\mathcal{L}^{n-1} - \mathcal{M}_s^{n+1}$ is defined to be the intersection number $(\mu^*\L-\M_s)\cdot (\sum_j \mu^*\L^j\cdot \M^{n-j}_s)$, where we view $\mu^*\L-\M$ is a divisor supported on $\Y_0$.
    This function is smooth with derivative
    \begin{align*}
        f'(s) = (n+1)E \cdot \mathcal{M}_s^n,
    \end{align*}
    which is strictly positive for all $0 < s < 1$ by ampleness. 
    Since the function satisfies $f(0) = 0$ and $f(1) = \mu^*\mathcal{L}^{n+1} - \mathcal{M}^{n+1}$, we thus obtain
    \begin{align*}
        \mu^*\mathcal{L}^{n+1}
        - \mathcal{M}^{n+1}
        > 0.
    \end{align*}
    
    On the other hand, the construction of the good filtration yields
    \begin{equation*}
        \mathcal{F}^0 H^0(\X_0,k\L_0)
        = \mathrm{im}(H^0(\mathcal{Y},k\mathcal{M}) \to H^0(\X_0,k\L_0)).
    \end{equation*}
    The kernel of the map is $H^0(\mathcal{Y},k\mathcal{M} - \hat{\X}_0)$, where $\hat{\X}_0$ is the strict transform of $\X_0$ along $\mu$.
    Since $H^0(\mathcal{Y},k\mathcal{M} - \hat{\X}_0) \neq H^0(\mathcal{Y},k\mathcal{M})$ for sufficiently large $k$, we have $\mathcal{F}^0 H^0(\X_0,k\L_0) \neq 0$ for sufficiently large $k$, hence there is a rational point $x$ in $P$ such that $G_\mathcal{F}(x) = 0$.
    
    Moreover, we have $G_\mathcal{F} \leq 0$ on the interior of $P$, since the filtration is concentrated in weight $\lambda \leq 0$.
    By concavity and constancy, this implies that $G_\mathcal{F}$ is identically zero almost everywhere.

    Combining Equation \eqref{eq : weight function as sum of gk functions} with the first part of the theorem and Equation \eqref{eq : coeff b_0 for models}, this shows that
    \begin{equation}
        \deg \big( \lambda_{n+1}(\mathcal{X},\mathcal{L}) - \lambda_{n+1}(\mathcal{Y},\mathcal{M}) \big)
        = 0.
    \end{equation}
    However, we have $\deg ( \lambda_{n+1}(\mathcal{X},\mathcal{L}) - \lambda_{n+1}(\mathcal{Y},\mathcal{M}) ) = \mu^*\mathcal{L}^{n+1} - \mathcal{M}^{n+1}$ using the interpretation of $\lambda_{n+1}$ in terms of Deligne pairings and a result of Boucksom--Eriksson \cite[Corollary 8.21 and A.23]{BE21}, a contradiction.
\end{proof}

Combining Theorem \ref{thm : hattoris good filtration properties} with Theorem \ref{thm : difference of CM degrees}, we obtain
\begin{equation}\label{eq : DF of Hattoris filtratoin = CM - CM}
    \mathrm{DF}(\mathcal{F})
    = \mathrm{CM}(\mathcal{X}',\mathcal{L}')
    - \mathrm{CM}(\mathcal{X},\mathcal{L}).
\end{equation}
In particular, together with the stability result in \cite[Theorem 2.18]{hattori2024minimizing}, one obtains Theorem \ref{thm : hattoris theorem} for models defined over DVRs that may not extend to a smooth projective curve.

\subsection{Twisting Hattori's good filtration.}\label{subsec : twisting hattoris good filtratoin}

In order to apply Theorem \ref{thm : realtive good filtration K stab of extr}, we examine twists of Hattori's good filtration.
Let $\beta$ be a $\G_{m,K}$-action on $(X,L)$ and assume that the models $(\mathcal{X},\mathcal{L})$ and $(\mathcal{X}',\mathcal{L}')$ are $\beta$-equivariant.
Choose a $\G_{m,R}$-equivariant resolution of indeterminacies $\mathcal{Y}$ of the equivariant birational map $\mathcal{X}' \dashrightarrow \mathcal{X}$, which makes the embedding in Equation \eqref{eq : embedding i} $\C^*$-equivariant.
Write $\beta_0$ for the $\C^*$-action $(\mathcal{X}_0,\mathcal{L}_0)$ obtained by restricting the $\G_{m,R}$-action on $(\mathcal{X},\mathcal{L})$ to the special fiber.
\begin{lemma}\label{lem : sq weight of hattoris good filtration}
    The filtration $\mathcal{F} = \mathcal{F}_{(\mathcal{X},\mathcal{L}),(\mathcal{X}',\mathcal{L}')}$ is $\beta_0$-invariant and satisfies
    \begin{equation*}
        w^2_{\mathcal{F},\beta_0}(k)
        = -\mathrm{wt}_\beta \big( H^0(\mathcal{X}, k\mathcal{L})/ i(H^0(\mathcal{Y},k\mathcal{M})) \big)
    \end{equation*}
    for $k \gg 0$, where $\mathrm{wt}_\beta ( H^0(\mathcal{X}, k\mathcal{L})/H^0(\mathcal{Y},k\mathcal{M}) )$ denotes the total weight of the $\C^*$-action on the quotient $H^0(\mathcal{X}, k\mathcal{L})/i(H^0(\mathcal{Y},k\mathcal{M}))$ induced by $\beta$.
\end{lemma}
\begin{proof}
    Note that $\beta_0$-invariance of $\mathcal{F}$ is immediate since $\mathcal{F}^\lambda H^0(\mathcal{X}_0,k\mathcal{L}_0)$ is the image $\C^*$-equivariant map
    \begin{equation*}
        \mathcal{F}^{\lambda}H^0(\mathcal{X},k\mathcal{L}) = \{ s \;|\; t^{-\lambda} s \in i(H^0(\mathcal{Y},k\mathcal{M})) \} \to H^0(\mathcal{X}_0,k\mathcal{L}_0).
    \end{equation*}
    
    Following the definition of the mixed weight-square function, it suffices to prove the identification
    \begin{equation*}
        (H^0(\mathcal{X},k\mathcal{L})
        / i(H^0(\mathcal{Y},k\mathcal{M})))_\mu
        \simeq \bigoplus_{\lambda \geq 0}  t^\lambda \big( H^0(\mathcal{X}_0,k\mathcal{L}_0)_\mu / (\mathcal{F}^{-\lambda} H^0(\mathcal{X}_0,k\mathcal{L}_0))_\mu \big),
    \end{equation*}
    for large $k$, where the subscripts $\mu$ in $\Z$ indicate weight-space components with respect to the $\C^*$-actions.
    This follows directly from Equation \eqref{eq : isomorphism for weight formula} as taking weights passes through direct sums and quotients.
\end{proof}
As in Section \ref{subsubsec: twisted and equivariant models}, the choice of uniformizer $t$ in $R$ induces an automorphism $\beta(t) : (X,L) \to (X,L)$. The same scheme $\mathcal{Y}$ fits into a new resolution of indeterminacies
\begin{equation*}
\begin{tikzcd}[row sep = 0.1em]
    & \mathcal{Y} \ar[dddddddddddddddddl,swap,"\mu'"] \ar[dddddddddddddddddr,"\mu_\beta"] &
    \\\\\\\\\\\\\\\\\\\\\\\\\\\\\\\\\\
    \mathcal{X}' \ar[rr,dashed,"\psi_\beta"] && \mathcal{X},
\end{tikzcd}
\end{equation*}
where $\psi_\beta : \mathcal{X}' \dashrightarrow \mathcal{X}$ is given by $\mathcal{X}'_K \simeq X \xrightarrow{\beta(t)} X \simeq \mathcal{X}_K$ and $\mu'$ is as before.
The birational morphism $\mu_\beta$ still satisfies $\mu^*_\beta \mathcal{L} = \mu'^*\mathcal{L}' + E - a \mathcal{Y}_0$ but gives rise to a \textit{twisted embedding}
\begin{equation*}
    i_\beta : H^0(\mathcal{Y},k\mathcal{M}) \to H^0(\mathcal{X},k\mathcal{L}).
\end{equation*}

\begin{lemma}
    We have $i_\beta = h_\beta \circ i$, where $h_\beta : H^0(\mathcal{X},k\mathcal{L}) \to H^0(\mathcal{X},k\mathcal{L})$ is the automorphism defined by
    \begin{equation*}
        h_\beta(s) =  \sum_{\mu \in \Z} t^{-\mu} s_\mu,
    \end{equation*}
    where we wrote $s = \sum_{\mu \in \Z} s_\mu$ with respect to the weight-space decomposition.
\end{lemma}
\begin{proof}
    Over $\Spec K$, we have $\mu_\beta = \phi_\beta \circ \mu$, where
    \begin{equation*}
        \phi_\beta : \mathcal{X}_K \simeq X \xrightarrow{\beta(t)} X \simeq \mathcal{X}_K.
    \end{equation*}
    This implies that $i(s) = \varphi^*_\beta i_\beta(s)$ pointwise on $\mathcal{X}_K$ for all sections $s$ in $H^0(\mathcal{Y},k\mathcal{M})$.
    The lemma then follows from the fact that the restriction of the automorphism $h_\beta$ to $H^0(\mathcal{X}_K,k\mathcal{L}_K)$ is induced by $\varphi_\beta^{-1}$.
\end{proof}
\begin{proposition}\label{prop : twisting hattoris filtratoin}
    The twist of $\mathcal{F} = \mathcal{F}_{(\mathcal{X},\mathcal{L}),(\mathcal{X}',\mathcal{L}')}$ by $\beta_0$ satisfies
    \begin{equation*}
        \mathcal{F}^\lambda_{\beta_0} H^0(\mathcal{X}_0,k\mathcal{L}_0)
        = \{ s|_{\mathcal{X}_0} \;|\; s \in H^0(\mathcal{X},k\mathcal{L}) \text{ and } t^{-\lambda}s \in i_\beta(H^0(\mathcal{Y},k\mathcal{M})) \}.
    \end{equation*}
    In particular, the twisted filtration $\mathcal{F}_{\beta_0}$ coincides with the good filtration associated to the twisted model $(\mathcal{X}_\beta,\mathcal{L}_\beta)$ and $(\mathcal{X}' ,\mathcal{L}' )$.
\end{proposition}
\begin{proof}
    Writing subscripts for the weight-space decomposition components, a direct computation shows that
    \begin{align*}
        \mathcal{F}^\lambda_{\beta_0} H^0(\mathcal{X},kr\mathcal{L})
        &= \bigoplus_{\mu \in \Z} \{ s|_{\mathcal{X}_0} \;|\; 
        s \in H^0(k\mathcal{L}) \text{ and } t^{\mu-\lambda}s \in i(H^0(kr\mathcal{M})) \}_\mu,
        \\
        &= \bigoplus_{\mu \in \Z} \{ s|_{\mathcal{X}_0} \;|\; 
        s \in H^0(k\mathcal{L})_\mu \text{ and } t^{\mu-\lambda} t^{-\mu} s \in i_\beta(H^0(kr\mathcal{M})) \},
        \\
        &= \bigoplus_{\mu \in \Z} \{ s|_{\mathcal{X}_0} \;|\; 
        s \in H^0(k\mathcal{L}) \text{ and } t^{-\lambda} s \in i_\beta(H^0(kr\mathcal{M})) \}_\mu,
        \\
        &= \{ s|_{\mathcal{X}_0} \;|\; 
        s \in H^0(k\mathcal{L}) \text{ and } t^{-\lambda} s \in i_\beta(H^0(kr\mathcal{M})) \},
    \end{align*}
    as desired.
\end{proof}
Lastly, we characterize positivity of the reduced norm of a good filtration induced by two equivariant models.
Let $T$ be a torus in $\mathrm{Aut}_K(X,L)$.
By Lemma \ref{lem : sq weight of hattoris good filtration}, if $(\mathcal{X},\mathcal{L})$ and $(\mathcal{X}',\mathcal{L}')$ are $T$-equivariant, then the good filtration $\mathcal{F} = \mathcal{F}_{(\mathcal{X},\mathcal{L}),(\mathcal{X}',\mathcal{L}')}$ is $T_0$-invariant, where $T_0$ in $\mathrm{Aut}(\mathcal{X}_0,\mathcal{L}_0)$ is the restriction of the torus $T_R$ in $\mathrm{Aut}_R(\mathcal{X},\mathcal{L})$ to the special fiber.
\begin{proposition}\label{prop : red norm of hattoris filtration}
    We have $||\mathcal{F}||_{T_0} > 0$ if and only if $(\mathcal{X}_\beta,\mathcal{L}_\beta) \not \cong (\mathcal{X}',\mathcal{L}')$ for all $\beta$ one-parameter subgroups $\beta$ of $T$.
\end{proposition}
\begin{proof}
    From Lemma \ref{lem : red norm is minimum}, we know that the infimum over $N_\R(T)$ in the definition of the reduced norm is attained by $\xi(\mathcal{F}) = \sum_{i = 1}^d \langle \mathcal{F},\beta_i \rangle/\langle \beta_i, \beta_i \rangle \beta_i$ where $(\beta_i)$ is any $\Z$-basis of $N_\Z(T)$.
    As $\mathcal{F}$ is good, the inner product $\langle \mathcal{F},\beta_i \rangle$ is a rational number, so $\xi(\mathcal{F})$ lies in $N_\Q(T)$.
    We obtain
    \begin{equation*}
        ||\mathcal{F}||_{T_0}
        = \min_{\xi \in N_\Q(T)}
        ||\mathcal{F}_\xi||,
    \end{equation*}
    where the twist $\mathcal{F}_\xi$ may be viewed as a $\Z$-filtration of $(\mathcal{X}_0,r\mathcal{L}_0)$ for a sufficiently divisible positive integer $r$.
    Applying Theorem \ref{thm : hattoris good filtration properties}, we obtain $||\mathcal{F}||_{T_0} > 0$ if and only if $(\mathcal{X}_\xi,\mathcal{L}_\xi) \not\cong (\mathcal{X}',\mathcal{L}')$ for all rational twists $\xi$ in $N_\Q(T)$.
    But this is equivalent to $(\mathcal{X}_\beta,\mathcal{L}_\beta) \not \cong (\mathcal{X}',\mathcal{L}')$ for all $\beta$ in $N_\Z(T)$ as we may simply scale a given automorphism.
\end{proof}

\subsection{Proof of the main results}

We start by restating the first main result, Theorem \ref{thm : main theorem}.
Let $T$ be a maximal torus in $\mathrm{Aut}_K(X,L)$.
\begin{theorem}
    Let $(\mathcal{X},\mathcal{L})$ and $(\mathcal{X}',\mathcal{L}')$ be $T$-equivariant models of $(X,L)$ such that the restricted torus $T_0$ in $\mathrm{Aut}(\mathcal{X}_0, \mathcal{L}_0)$ is maximal. If $(\mathcal{X}_0, \mathcal{L}_0)$ smooth and extremal, then there exists a rational one-parameter subgroup $\xi$ of $T$, such that
    \begin{equation*}
        \mathrm{CM}(\mathcal{X}_\xi,\mathcal{L}_\xi)
        \leq 
        \mathrm{CM}(\mathcal{X}',\mathcal{L}'),
    \end{equation*}
    where the inequality is strict if and only if $(\mathcal{X}_\beta,\mathcal{L}_\beta) \not \cong (\mathcal{X}',\mathcal{L}')$ for all one-parameter subgroups $\beta$ of $T$.
\end{theorem}
\begin{proof}
    By invariance of the difference of CM degrees and the Donaldson--Futaki invariant under scaling, we may assume that $L$, $\mathcal{L}$ and $\mathcal{L}$ are all line bundles.
    The good filtration $\mathcal{F} = \mathcal{F}_{(\mathcal{X},\mathcal{L}),(\mathcal{X}',\mathcal{L}')}$ is $T_0$-invariant and, following Proposition \ref{prop : twisting hattoris filtratoin} and Equation \eqref{eq : DF of Hattoris filtratoin = CM - CM}, we have
    \begin{equation*}
        \mathrm{DF}_{T_0}(\mathcal{F})
        = \mathrm{DF}(\mathcal{F}_\xi)
        = \mathrm{CM}(\mathcal{X}',\mathcal{L}')
        - \mathrm{CM}(\mathcal{X}_\xi,\mathcal{L}_\xi),
    \end{equation*}
    for $\xi = - \sum_{i = 1}^d \langle \mathcal{F}, \beta_i \rangle / \langle \beta_i, \beta_i \rangle \beta_i$, where $(\beta_i)$ is any of $\Z$-basis of $N_\Z(T)$.
    As $(\mathcal{X}_0,\mathcal{L}_0)$ is extremal, we have $\mathrm{DF}_{T_0}(\mathcal{F}) \geq 0$ with strict inequality if and only if $||\mathcal{F}||_{T_0} > 0$ by Theorem \ref{thm : realtive good filtration K stab of extr}.
    The result follows by Proposition \ref{prop : red norm of hattoris filtration}.
\end{proof}
\begin{remark}\label{rmk : xi F is explicit}
    The proof reveals that we may explicitly give the rational one-parameter subgroup $\xi$ as $-\xi(\mathcal{F}) = -\sum_{i = 1}^d \langle \mathcal{F},\beta_i \rangle/\langle \beta_i, \beta_i \rangle \beta_i$ from Lemma \ref{lem : red norm is minimum}. Geometrically, it may be viewed as the orthogonal projection of the filtration $\mathcal{F}$ to $N_\Q(T)$ whose coefficients may be expressed geometrically in terms of Hilbert- and weight polynomials of the models, following Theorem \eqref{thm : hattoris good filtration properties} and Lemma \ref{lem : sq weight of hattoris good filtration}.
\end{remark}

By replacing $\mathrm{DF}_{T_0}(\mathcal{F})$ by $\mathrm{DF}(\mathcal{F})$ and applying Corollary \eqref{cor : good filtration k polystab of csck} with Theorem \ref{thm : hattoris good filtration properties} (which uses the vanishing of the Futaki invariant for cscK manifolds \cite{futaki1983obstruction}), we obtain Corollary \ref{cor : main corollary} for cscK metrics.





\printbibliography

\end{document}